\documentclass[a4paper,11pt]{article}
\usepackage{pdfsync}

\usepackage[plainpages=false]{hyperref}
\usepackage{amsfonts,amsmath,amssymb,amsthm}
\usepackage{latexsym,lscape,rawfonts,mathrsfs}
\usepackage[utf8]{inputenc}
\usepackage[overload]{empheq}
\usepackage{cases}
\usepackage[dvips]{color}
\usepackage{multicol}
\usepackage{extarrows}

\usepackage[all]{xy}
\usepackage{eufrak}
\usepackage{makeidx}
\usepackage{graphicx,psfrag}
\usepackage{pstool}

\usepackage{array,tabularx}

\usepackage{setspace}

\usepackage[titletoc,title]{appendix}

\usepackage{txfonts}

\usepackage{makeidx}

\usepackage{tikz} 
\usetikzlibrary{positioning}
\usepackage{tikz-cd}
\usepackage{tocloft}

\newcommand{\R}{\ensuremath{\mathbb{R}}}

\newcommand{\ba}{\begin{align*}}
	\newcommand{\ea}{\end{align*}}

\makeatletter
\def\ExtendSymbol#1#2#3#4#5{\ext@arrow 0099{\arrowfill@#1#2#3}{#4}{#5}}

\makeatother

\makeatletter
\def\ExtendSymbol#1#2#3#4#5{\ext@arrow 0099{\arrowfill@#1#2#3}{#4}{#5}}

\makeatother

\definecolor{hao}{rgb}{1,0.5,0}
\definecolor{miao}{cmyk}{0.5,0,0.2,0.2}
\definecolor{qiao}{gray}{0.96}

\newtheorem{prop}{Proposition}[section]

\newtheorem{proposition}[prop]{Proposition}

\newtheorem{theorem}[prop]{Theorem}

\newtheorem{lemma}[prop]{Lemma}

\newtheorem{corollary}[prop]{Corollary}

\newtheorem{remark}[prop]{Remark}

\numberwithin{equation}{section}

\makeindex

\title{Heat kernel on Ricci shrinker metric measure spaces}
\author{Bing Wang\; and Jie Wang\footnote{Corresponding author.}}
\date{}

\begin{document}
	\maketitle	
\begin{abstract}
As a metric measure space possessing positive Bakry-\'{E}mery curvature, a Ricci shrinker with potential function $f$ admits a heat kernel $H_f$ under the weighted volume measure $e^{-f}dv$. In this paper, we study $H_f$ systematically and establish a series of fundamental estimates. Moreover, we clarify the equivalence between $H_f$ and the spacetime heat kernel $H(x,t;y,s)$ under Ricci flow induced by a Ricci shrinker. Inspired by this relation, we extend corresponding analysis tools of Ricci flows to Ricci shrinker metric measure spaces, such as Nash entropy and reduced distance. As a direct application, we prove that $\left|\nabla R\right|=o(f^{3/2})$ implies $R=o(f)$.
\end{abstract}

\tableofcontents
\setlength{\cftbeforesecskip}{3pt}
	
\section{Introduction}	
Let $(M^n,g)$ be a  complete $n$-dimensional Riemannian manifold and $Rm,Rc,R$ be the Riemannian curvature, Ricci curvature and scalar curvature respectively. $(M^n,g)$ is called a gradient Ricci soliton if there exists a smooth function $f$ and scalar $\lambda$ such that
$$Rc+Hess\hspace*{0.1em} f=\lambda g.$$
A Ricci soliton is called shrinking, steady or expanding,  if $\lambda>0,\lambda=0$ or $\lambda<0$ respectively. By scaling the metric $g$, we may assume $\lambda\in \left\lbrace \frac{1}{2},0,-\frac{1}{2}\right\rbrace$. Especially, a Ricci shrinker is a triple $(M^n, g,f)$ of smooth manifold $M^n$, Riemannian metric $g$ and a smooth function $f$ satisfying
\begin{equation}\label{1.1*}
	Rc+Hess\hspace*{0.1em}f=\frac{1}{2}g.
\end{equation}
Tracing (\ref{1.1*}) gives
\begin{equation}\label{1.2‘}
	R+\Delta f=\frac{n}{2}.
\end{equation} 
By a normalization of $f$, we may always assume
\begin{align}
	R+\left|\nabla f\right|^2&=f,\label{1.3‘}\\
	\int_{M^n}(4\pi)^{-\frac{n}{2}}e^{-f}dv&=e^{\boldsymbol{\mu}},\label{1.4‘}
\end{align}
where $\boldsymbol{\mu}\leq0$ is the functional of Perelman(cf.\cite{LW}) and $dv$ is the standard volume measure.	Moreover, by \cite{HM}, there exists a point $p\in M^n$ where $f$ attains its infimum. Besides, see also \cite{HLW,LLW,WangW} for further geometric properties of Ricci shrinkers.

Ricci shrinkers arise as self-similar ancient solutions to Ricci flows and play a key role in the singularity analysis. Since the significant work of Perelman \cite{P}, heat kernel becomes the fundamental analysis tool in this field and shows increasing importance. In the past few years, the first named author and Y. Li have made substantial progress in the study of heat kernel on Ricci shrinkers and achieved the complete classification of all K\"{a}hler Ricci shrinker surfaces, see \cite{LW,LW2,LW3,Ba}. Among these works, the heat kernel means the natural positive spacetime function $H(x,t;y,s)$. More precisely, let $(M^n,g(t))_{t<1}$ be the ancient Ricci flow associated with a Ricci shrinker,  then there exists a positive heat kernel function $H(x,t;y,s)$ for all $x,y\in M^n$ and $s<t<1$ such that 

$$\square H(\cdot,\cdot;y,s)=0,\quad\lim_{t\searrow s}H(\cdot,t;y,s)=\delta_y$$
and
$$\square^*H(x,t;\cdot,\cdot)=0,\quad\lim_{s\nearrow t}H(x,t;\cdot,s)=\delta_x,$$
where $\square:=\partial_t-\Delta_{g(t)}$ and $\square ^* : = - \partial _t- \Delta_{g(s)} + R_{g(s)}.$ Furthermore, the heat kernel $H$ satisfies the semigroup property
$$H(x,t;y,s)=\int_{M^n}H(x,t;z,r)H(z,r;y,s)\:\mathrm{d}v_{g(r)}(z),\quad\forall\:x,y\in M^n,\:r\in(s,t)\subset(-\infty,1),$$
and 
\begin{align*}
	\int_{M^n}H(x,t;y,s)\:\mathrm{d}v_{g(t)}(x)\leq1,\quad\int_{M^n}H(x,t;y,s)\:\mathrm{d}v_{g(s)}(y)=1.
\end{align*}

On the other hand, in light of the Ricci shrinker equation (\ref{1.1*}),  any such metric also belongs to the metric measure space $CD\left(\frac{1}{2},\infty\right)$ and possesses positive Bakry-\'{E}mery curvature, i.e. $Rc_f:=Rc+Hess f=\frac{g}{2}>0$. Under its natural weighted measure $e^{-f}dv$ with Laplacian operator $\Delta_f:=\Delta-\left\langle\nabla f,\nabla \cdot\right\rangle$, there exists a symmetric complete heat kernel $H_f(x,y,t)$ related to the heat operator $\partial_t-\Delta_f$ (cf. \cite{AG2}), i.e.
\begin{align*}
	\left(\partial_t-\Delta_f\right)H_f=0,\quad \lim\limits_{t\longrightarrow0}H_f(x,y,t)\longrightarrow\delta_y,\quad \int_{M^n}H_f(x,y,t)e^{-f(y)}dv(y)=1.
\end{align*}
In this setting, there are still quite a few results, see \cite{HLW,WW} and references therein. However, most related heat kernel estimates deduced by conventional methods are pretty rough since corresponding weighted Sobolev and Poinc\'{a}re inequalities are not very sharp. By Davies \cite{Da2}, we know Log-Sobolev inequality with variable parameters is better suited  to attain the  ultracontractivity of heat kernel. For Ricci shrinkers, this type of inequality is proved by Li-Wang \cite{LW} and is applied to derive the upper bound of $H(x,t;y,s)$. However, there is no such Log-Sobolev inequality with variable parameters under the weighted measure $e^{-f}dv$. Therefore, it's natural to transform $H_f$ to some function which can be estimated similarly under the standard measure $dv$.

Via Doob's transform which means measure transform and is quite common in stochastic analysis, $H_f$ turns into the Schrödinger heat kernel $H_\phi$ for heat operator $\partial_t-(\Delta-\phi)$ under standard measure $dv$, where $\phi=(f+R-n)/4.$
Especially, it holds that
\begin{align*}
	H_f=H_\phi e^{\frac{f(x)+f(y)}{2}}.
\end{align*} 
For details, see section \ref{sec-pe}. From this point of view, we have
\begin{theorem}\label{t1.1'}
	For any Ricci shrinker $(M^n,g,f)$, $x,y\in M^n$ and $t>0$,
	\begin{align*}
		H_\phi(x, y, t) \leq \frac{C(n)e^{-\boldsymbol{\mu}}}{\left(1-e^{-t} \right)^{\frac n2}} \exp \left(-\frac{f(x)+f(y)}{2} \tanh \frac{t}{4}\right),
	\end{align*}
	where $C(n)$ is a dimensional constant. Hence by $H_{f}=H_{\phi} e^{\frac{f(x)+f(y)}{2}}$,
	\begin{align*}
		H_f(x, y, t) \leq \frac{C(n)e^{-\boldsymbol{\mu}}}{\left(1-e^{-t} \right)^{\frac n2}}\exp \left(\frac{f(x)+f(y)}{2} \left(1-\tanh \frac{t}{4}\right)\right).
	\end{align*}
\end{theorem}

Note that on $\left(\R, g, e^{-\frac{x^2}{4}}dx\right)$,

\begin{align*}
	H_{f}^{\R}\left( x,y,t\right)=\frac{1}{\left(4\pi (1-e^{-t}) \right)^{\frac12}}
	\exp \left( -\frac{(x-y)^2}{8 \sinh \frac{t}{2}} +\frac{x^2+y^2}{8}\left[ 1-\tanh \frac{t}{4} \right] \right).
\end{align*}	
Hence the exponential term in our estimate of Theorem \ref{t1.1'} is optimal on diagonal $x=y$. However, if $d(x,y)$ is very large, the above upper bound is not too precise for the lack of distance term. To improve this issue, we establish the following estimate for $x=p$.

\begin{theorem}\label{t1.2'}
	For any Ricci shrinker $(M^n,g,f)$, $x\in M^n$, $\delta\in\left(\frac{1}{2}, \frac{2}{3}\right)$ and $t\geq \frac{2}{3\delta}$, there exists a constant $C=C(n,\delta)$ such that 
	\begin{equation*}
		H_f(p,x,t)\leq Ce^{\boldsymbol{-\mu}}e^{-\frac{e^{-2\delta t}}{1-e^{-2\delta t}}f(x)}.
	\end{equation*}
	For $0<t\leq\frac{4}{3}$ and $4<D<6$,
	\begin{equation*}
H_f(p,x,t)\leq\frac{C(n,D)e^{-\boldsymbol{\mu}}}
{\left(1-e^{-t}\right)^{\frac n2}}\exp\left(-\frac{d^2(p,x)}{Dt}
+\frac{f(x)}2\right).
	\end{equation*}
\end{theorem}	
Here for the first case, we remark that on $\left(\R, g, e^{-\frac{x^2}{4}}dx\right)$, $\delta=\frac{1}{2}$, so this estimate is almost sharp. For the latter case, since $Dt<8$ and $f(x)$ increases quadratically with respect to the distance $d(p,x)$ (i.e. $f(x)\sim d^2(p,x)/4$, cf. \cite{HM}), so 
$$-\frac{d^2(p,x)}{Dt}+\frac{f(x)}{2}\sim -\left(\frac{1}{Dt}-\frac{1}{8}\right)d^2(p,x)\quad \mbox{as}\quad d(p,x)\longrightarrow\infty.$$
In other words, with the help of a gradient estimate of $H_f$, we can get  another type of upper bound based on Theorem \ref{t1.2'} for arbitrary $x,y\in M^n$ which is different from that of Theorem \ref{t1.1'}. Besides, during the arguments to prove Theorem \ref{t1.1'} and \ref{t1.2'}, many other corresponding estimates are also established.

Specifically, for spacetime heat kernel upper bound on Ricci shrinkers and $0<\sigma<1$, Li--Wang \cite[Theorem~1.3]{LW2} proves that if $d_t(p,x)\le\sigma^{-1}$ and $-\sigma^{-1}\leq s<t\leq1-\sigma$, if $\boldsymbol{\mu}\geq-A$, then
\begin{align*}
	(i)\quad H(x,t;y,s)
	&\le
	\frac{C(n,A,\sigma)}{(t-s)^{\frac n2}}
	\exp\left\{
	-\frac{d_s^2(x,y)}{C(n,A,\sigma)(t-s)}
	\right\}.
\end{align*}
 Meanwhile, let $\mathcal N_{x,t}(t-s)$ be the pointed Nash entropy, then for an \(H_n\)-center $(z,s)$ of $(x,t)$ and $\varepsilon>0$, 
\begin{align*}
	(ii)\quad H(x,t;y,s)
	\le
	\frac{C(n,\varepsilon)e^{-\mathcal N_{x,t}(t-s)}}
	{(t-s)^{\frac n2}}
	\exp\left\{
	-\frac{d_s^2(z,y)}{(4+\varepsilon)(t-s)}
	\right\}.
\end{align*}
 By the exact correspondence between $H_f$ and the spacetime heat kernel \(H(x,t;y,s)\)
established in Proposition \ref{p6.1}, Proposition \ref{p7.6} then yields, for every $\varepsilon>0$, $0<\beta<1$, $s<t<1$ and $x,y\in M^n$,
\begin{align*}
	&H(x,t;y,s)\\
	&\le
	\frac{C(n,\varepsilon)e^{-\boldsymbol{\mu}}}{(t-s)^{\frac n2}}
	\exp\left\lbrace
	\frac{1-\beta\tanh\left(\frac14\ln\frac{1-s}{1-t}\right)}{2}
	f\left(\psi_{\ln\frac1{1-t}}(x)\right)-\frac{1+\beta\tanh\left(\frac14\ln\frac{1-s}{1-t}\right)}{2}f\left(\psi_{\ln\frac1{1-s}}(y)\right)\right. \\
	&\left. \quad-\frac{(1-\beta)d^2\left(\psi_{\ln\frac1{1-t}}(x),\psi_{\ln\frac1{1-s}}(y)\right)}{(4+\varepsilon)\ln\frac{1-s}{1-t}}
	 \right\rbrace.
\end{align*}
In contrast, the spacetime form of Proposition \ref{p7.6} is valid for all $s<t<1$ and all $x,y\in M^n$, including terminal points escaping to infinity and times approaching the singular time.  Moreover, it
simultaneously retains Gaussian decay in the explicit distance and endpoint-potential weights.  The parameter $\beta$ describes the trade-off between these two effects: small $\beta$ preserves
the nearly Euclidean Gaussian coefficient, whereas large $\beta$ increases the potential-weight component.  Li--Wang \cite[Theorem~1.3(ii)]{LW2} obtains for an arbitrary terminal point, a global Gaussian estimate with the nearly Euclidean denominator $4+\varepsilon$, but its Gaussian factor is centered at an $H_n$-center whose location is not prescribed geometrically. To be precise, these two type of estimates are complementary rather than one being obviously superior to the other.  The advantages of our estimates are their global spacetime validity and their direct potential decay, rather than a sharper Gaussian constant.

Any Ricci shrinker metric $g$ induces an ancient solution to Ricci flow by the diffeomorphisms generated by $\frac{\nabla_g f}{1-t}$ for $t<1$ (cf. \cite[Chapter 4]{CLN}), so in principle, $H_f$ should be equivalent to $H(x,t;y,s)$ up to diffeomorphisms and scaling. To this end, in section \ref{sec-diff}, we discuss the relation between these two types of heat kernels in detail. Simply put, by appropriate
inverse diffeomorphisms	and scaling, the forward heat kernel $H(\cdot, \cdot,y,s)$ becomes $H_f(\cdot, y,t)$ and the backward one $H(x,t,\cdot,\cdot)$ turns into $H_f(x,\cdot,s)e^{-f(\cdot)}$. In other words, this type of transform not only provides a new perspective to study the forward heat kernel $H(\cdot, \cdot; y,s)$ but also clarifies the relation between it and the backward $H(x,t,\cdot,\cdot)$. Besides, we are able to extend the analysis tools for Ricci flow to Ricci shrinker metric measure space $\left(M^n,g,e^{-f}dv\right)$, such as Perelman's entropy, reduced distance, etc.

Especially, to derive the lower bound estimate of $H_f(x, y, t)$, we introduce the weighted reduced distance $\bar{\ell}^x_f(y,t)$. In fact, we have that $\ell^x_f=\bar{\ell}^x_f+f$ is the diffeomorphic counterpart of Perelman's reduced distance under Ricci flow. Therefore, $\bar{\ell}^x_f$ can be regarded as a kind of forward reduced distance in the Ricci shrinker metric measure spaces.

\begin{theorem}\label{t1.3'}
	For any Ricci shrinker $(M^n,g,f)$, $x,y\in M^n$ and $t>0$,
	\begin{align*}
		H_f(x, y, t) \geq \frac{e^{-\bar{\ell}^x_f(y,t)}}{\left(4\pi (1-e^{-t}) \right)^{\frac n2}}.
	\end{align*}
	Moreover, if the inequality becomes an equality somewhere, then the Ricci shrinker is isometric to the standard flat Gaussian shrinker $\left(\mathbb{R}^n,g_E,\frac{x^2}{4}\right)$. 
\end{theorem}

Besides, by diffeomorphisms and rescaling, Perelman's differential Harnack inequality of Ricci shrinkers (cf. \cite[Theorem 21]{LW}) has a static version with respect to $H_f$. Based on this type of inequality and our heat kernel estimates, we have
\begin{theorem}\label{t1.4'}
	Let $(M^{n},g,f)$ be a complete noncompact Ricci shrinker. Assume that $|\nabla R|=o\left(f^{3/2}\right)$ uniformly at infinity, i.e.
	\begin{align*}
		\eta(L):=\sup_{\{z\in M^n:\,f(z)\ge L\}}\frac{|\nabla R|(z)}{f^{3/2}(z)}\longrightarrow0
		\quad\text{as }L\longrightarrow\infty.
	\end{align*}
	Then
	\begin{align*}
		\frac{R(x)}{f(x)}\longrightarrow0\quad\text{as }f(x)\longrightarrow\infty.
	\end{align*}
\end{theorem}	

\subsection*{Organization of the paper}	
The paper is organized as follows. In Section 2, we review the measure transform, introduce the Schrödinger kernel $H_\phi$ and its relation with $H_f$, then give the explicit weighted heat kernel of Gaussian shrinker. In Section 3, we prove the ultracontractivity and Gaussian upper bound for $H_\phi$. Sections 4 and 5 establish the potential weighted upper bounds and the base-point decay estimates. In Section 6, by diffeomorphisms, we derive the exact relation between $H_f$ and the spacetime heat kernel of the self-similar Ricci flow induced by a Ricci shrinker. Section 7 introduces the pointed Nash and $\mathcal{W}$-entropies on Ricci shrinkers, proves their monotonicity, and establishes their long-time convergence to $\boldsymbol{\mu}$. In Section 8, we pull back Perelman’s reduced distance under Ricci flow to Ricci shrinkers, and define corresponding weighted reduced distance, then obtain the heat kernel lower bound by diffeomorphisms again. Finally, Section 9 applies the differential Harnack inequality to the geometry of scalar curvature, especially the quantitative low curvature selection, and proves $\left|\nabla R\right|=o(f^{3/2})\Longrightarrow R=o(f)$.

\section{Preliminaries}\label{sec-pe}

Let $H_f$ be the heat kernel associated with the weighted Laplace operator
\begin{align*}
	\Delta_f= \Delta - <\nabla f, \nabla\cdot>,
\end{align*}
and $H_{\phi}$ be the kernel relating to
\begin{align*}
	\Delta - \phi = \Delta - \frac{1}{4} \left(f+R-n\right),
\end{align*}
where $\phi=\frac{1}{4} \left(f+R-n\right)$ satisfies $\Delta e^{-\frac{f}{2}}= \phi e^{-\frac{f}{2}}$.
On a Ricci shrinker $(M^n,g,f)$, we have the identities
\begin{align}\label{1.1}
	R+ \Delta f =\frac{n}{2}, \quad R+|\nabla f|^2=f. 
\end{align}
By (\ref{1.1}), it's easy to check
\begin{align*}
	\Delta e^{-\frac{f}{2}}=-\frac{1}{4} \left\{2\Delta f -|\nabla f|^2 \right\} e^{-\frac{f}{2}} =\frac{1}{4} \left\{R+f-n \right\} e^{-\frac{f}{2}}. 
\end{align*}
As $\phi=\frac{1}{4} \left\{R+f-n \right\}$, the above equation is the same as
\begin{align*}
	\left( \Delta - \phi \right) e^{-\frac{f}{2}}=0. 
\end{align*}
Suppose $u$ satisfies $\Delta_f u=\lambda u$. Then we have
\begin{align*}
	(\Delta-\phi) \left\{ e^{-\frac{f}{2}} u\right\}&=\left\{(\Delta-\phi) e^{-\frac{f}{2}} \right\} u + e^{-\frac{f}{2}}  \Delta u + 2 \langle \nabla e^{-\frac{f}{2}}, \nabla u\rangle\\
	&=e^{-\frac{f}{2}}  \Delta u -e^{-\frac{f}{2}} \langle \nabla f, \nabla u\rangle\\
	&=e^{-\frac{f}{2}} \Delta_f u= \lambda \left\{ e^{-\frac{f}{2}} u \right\}. 
\end{align*}
Formally, by the eigenfunction expansion of heat kernel, we have 
\begin{align*}
	H_{\phi}=H_{f} e^{-\frac{f(x)+f(y)}{2}}.
\end{align*}
Clearly, since $H_f\longrightarrow \delta_x$ as $t\longrightarrow0$ under the measure $e^{-f}dv$, so $H_\phi$ also tends to $\delta_x$ function as $t\longrightarrow0$ under the standard volume measure $dv$. 

In order to understand further this type of heat kernel, we may change the measure $d\sigma=dx$ or the operator $\Delta_\sigma$ to a better one. In fact, the idea of changing measure or operator originates from stochastic analysis and the well-known method is called the Doob's transform. 

Let $\mathcal{D}:=C^\infty_0(M^n)$ and by a self-adjoint extension, $\Delta_\sigma|_{\mathcal{D}}$ can be regarded as an operator on $W^2_0(M^n, g,d\sigma)$, and $\Delta_{W^2_0}$ is a self-adjoint non-positive definite operator in $L^2(M^n,g,d\sigma)$. Note that here the space $W^2_0$ is not the usual Sobolev space $H^2_0$, see \cite[section 2.2]{AG2}. Any real-valued function $\phi\in L^2_{loc}(M^n,g,d\sigma)$ can be considered as a multiplication operator in $L^2(M^n,g,d\sigma)$ with the domain $\mathcal{D}$, and clearly, $-\Delta_\sigma+\phi$ is a symmetric operator on $\mathcal{D}$. Moreover, assume that a smooth positive function $h$ on $M^n$ satisfies the equation
\begin{equation}\label{1.2*}
	\Delta_\sigma h-\phi h=0,
\end{equation}
then the operator $\left(-\Delta_\sigma+\phi\right)|_{\mathcal{D}}$ in $L^2(M^n,g,d\sigma)$ admits a self-adjoint extension $\Phi$, and this extension is unique provided $(M^n,g)$ is complete. Let $h$ be a smooth positive function satisfying (\ref{1.2*}) and $d\tilde{\sigma}=h^2d\sigma$, then by \cite[Lemma 4.3]{AG2}, 
$$\Delta_{\tilde{\sigma}}=h^{-1}\circ(\Delta_\sigma-\phi)\circ h.$$
Hence for any $u\in \mathcal{D}$, there holds
$$\left(u,(-\Delta_\sigma+\phi)u\right)_{L^2(d\sigma)}=-\left(\frac{u}{h},\Delta_{\tilde{\sigma}}\left(\frac{u}{h}\right)\right)_{L^2(d\tilde{\sigma})}.$$ 
Applying Green's formula to both sides gives
\begin{equation}\label{1.3*}
	\int_{M^n}\left(\left|\nabla u\right|^2+\phi u^2\right)d\sigma=\int_{M^n}\left|\nabla(h^{-1}u)\right|^2d\tilde{\sigma}\geq 0,
\end{equation}
which implies the operator $-\Delta_\sigma+\phi$ is non-negative definite. Strictly, we have

\begin{theorem}[\cite{AG2}, Lemma 4.7]\label{p1.1}
	Let $h$ be a smooth positive function satisfying $\Delta_\sigma h= \phi h$, then the heat kernel $H_\phi(x,y,t)$ of the heat semigroup $e^{-t\Phi}$ in $L^2(M^n,g,d\sigma)$ and the heat kernel $H_{\tilde{\sigma}}(x,y,t)$ of $e^{t\Delta_{\tilde{\sigma}}}$ are related by
	\begin{equation*}\label{1.4}
		H_\phi(x,y,t)=h(x)h(y)H_{\tilde{\sigma}}(x,y,t).
	\end{equation*}
\end{theorem}
As we explained previously, in this paper, we just take $h(x)=e^{-\frac{f(x)}{2}}$. 	For details of this part, especially the existence of $H_f$, see \cite{AG2}.

For example, on one dimensional $\left(\mathbb{R}, g_{E}, e^{-\frac{x^2}{4}}dx\right)$, $f=\frac{x^2}{4}$ and the weighted Laplace operator is $\Delta_f=\frac{d^2}{dx^2}-\frac{x}{2}\frac{d}{dx}$. It is easy to check 
$$h_k''-\frac{x}{2}h_k'+\frac{k}{2}h_k=0,$$ 
hence the eigenvalues and corresponding eigenfunctions are $\left\lbrace- \frac{k}{2}\right\rbrace_{k=0}^{\infty}$ and Hermite polynomials
$$h_k(x)=(-1)^ke^{\frac{x^2}{4}}\frac{d^k}{dx^k}e^{-\frac{x^2}{4}},$$
which satisfy
\begin{align*}
	\int_{\mathbb{R}}h_j(x)h_k(x)e^{-\frac{x^2}{4}}dx
	=\left\{\begin{array}{ccc}
		0, & j\neq k, \\
		\frac{k!\sqrt{\pi}}{2^{k-1}}, & j=k. 
	\end{array}\right.
\end{align*}
Due to the completeness of Hermite polynomials in $L^2\left(\mathbb{R},g,e^{-\frac{x^2}{4}}dx\right)$, the weighted heat kernel $H_f(x,y,t)$ on $\left(\mathbb{R},g,e^{-\frac{x^2}{4}}dx\right)$ is given by (see \cite[Theorem 7.13]{Da} for the computational method)
\begin{align*}
	H_f(x,y,t)&=\sum_{k=0}^{\infty}e^{-\frac{kt}{2}}\frac{2^{k-1}h_k(x)h_k(y)}{k!\sqrt{\pi}}\\
	&=\frac{e^{\frac{t}{4}}}{\left(8\pi\sinh{\frac{t}{2}}\right)^{\frac{1}{2}}}\exp\left(\frac{2xye^{-\frac{t}{2}}-(x^2+y^2)e^{-t}}{4(1-e^{-t})}\right)\\
	&=\frac{1}{\left(4\pi (1-e^{-t}) \right)^{\frac12}}
	\exp \left( -\frac{(x-y)^2}{8 \sinh \frac{t}{2}} +\frac{x^2+y^2}{8}\left[ 1-\tanh \frac{t}{4} \right] \right).
\end{align*}
Consequently, we have
\begin{align*}
	H_{\phi}(x,y,t)&=H_{f}\left( x,y,t\right) e^{-\frac{x^2+y^2}{8}}\\
	&=\frac{1}{\left(4\pi (1-e^{-t}) \right)^{\frac12}}
	\exp \left\{ -\frac{(x-y)^2}{8 \sinh \frac{t}{2}} -\frac{x^2+y^2}{8}\tanh \frac{t}{4} \right\}\\
	&=\frac{e^{\frac{t}{4}}}{(8\pi\sinh \frac{t}{2})^{\frac12}} 
	\exp \left\{ -\frac{(x-y)^2}{8 \sinh \frac{t}{2}} -\frac{x^2+y^2}{8}\tanh \frac{t}{4} \right\}. 
\end{align*}

Furthermore, for any Ricci shrinker $(M^n,g,f)$, its weighted volume $V_f(M^n,g)=	\int_{M^n}e^{-f}dv$ is finite, especially, if we normalize $f$ as usual, then $V_f(M^n,g)=(4\pi)^{\frac{n}{2}}e^{\boldsymbol{\mu}}$, where the $\boldsymbol{\mu}$ is the Perelman functional. Note that $V_f<\infty$ implies the weighted heat kernel $H_f$ is complete, i.e., $H_fe^{-f}dv$ is a probability measure. As a consequence of \cite[Example 3.19]{AG2} and Theorem \ref{p1.1}, we have
\begin{proposition}\label{p1.2}
	For any Ricci shrinker $(M^n,g,f)$, the weighted heat kernel 
	\begin{equation*}
		H_f(x,y,t)\longrightarrow (4\pi)^{-\frac{n}{2}}e^{-\boldsymbol{\mu}} \hspace*{0.5em}as\hspace*{0.5em}t\longrightarrow\infty,
	\end{equation*}
	and the Mehler heat kernel
	\begin{equation*}
		H_\phi(x,y,t)\longrightarrow (4\pi)^{-\frac{n}{2}}e^{-\boldsymbol{\mu}-\frac{f(x)+f(y)}{2}} \hspace*{0.5em}as\hspace*{0.5em}t\longrightarrow\infty,
	\end{equation*}
	where we set $h=e^{-\frac{f}{2}}$ in Theorem \ref{p1.1}.
\end{proposition}

\section{Ultracontractivity}
First we show the ultracontractivity of $H_\phi$.
\begin{theorem}\label{t2.1}
	For any Ricci shrinker $(M^n,g,f)$ and $x,y\in M^n,t>0$,
	\begin{equation}\label{2.6*}
		H_\phi(x,y,t)\leq \gamma(t):=\left\{
		\begin{array}{lr}
			(4\pi t)^{-\frac{n}{2}}e^{-\boldsymbol{\mu}+\frac{n}{4}t} ,&t\leq1,\\
			(4\pi )^{-\frac{n}{2}}e^{-\boldsymbol{\mu}+\frac{n}{4}} ,&t>1.
		\end{array}
		\right.	
	\end{equation}
\end{theorem}
\begin{proof}
	Since the heat kernel $H_\phi(x,y,t)$ can be approximated by $H_\phi^{\Omega_i}(x,y,t)$ for any compact exhaustion $\left\lbrace\Omega_i \right\rbrace_{i=0}^\infty$ of $M^n$, hence we only need to prove (\ref{2.6*}) for any compact $\Omega \subset M^n$ with smooth boundary. Besides this point of view, we can also adopt standard cut-off arguments as in Li-Wang \cite{LW} instead of approximating.
	
	The following idea of proving the ultracontractivity via Log-Sobolev inequality originates from \cite{Da2}.	Let $u(y,0)$ be any non-negative smooth function with compact support in $\Omega$, then for any $T>0$ and $t\in\left[0,T\right]$,
	$$u(x,t)=\int_{\Omega}u(y,0)H_\phi^{\Omega}(x,y,t)dv(y) \in C^\infty(\Omega)\bigcap W^2_0(\Omega)$$
	is a solution to the heat equation
	\begin{equation}\label{2.7*}
		\partial_t u=(\Delta-\phi)u.
	\end{equation}
	Set 
	$$\left\|u\right\|_{p(t)}=\left(\int_{\Omega}\left|u\right|^{p(t)}\right)^{\frac{1}{p(t)}},$$
	where $p(t)=\frac{T}{T-t}$, $t\in\left[0,T\right]$, which implies $p(0)=1$ and $p(T)=\infty$.
	By (\ref{2.7*}), we compute
	\begin{align*}
		\partial_t\left\|u\right\|_{p(t)}=-\frac{p'}{p^2}\left\|u\right\|_{p(t)}\ln\left(\left\|u\right\|_{p(t)}^{p(t)}\right)+\frac{\left\|u\right\|_{p(t)}^{1-p(t)}}{p(t)}\left[p'\int_{\Omega}u^p\ln udv+p\int_{\Omega}u^{p-1}(\Delta u-\phi u)dv\right].
	\end{align*}
	Multiplying by $p^2\left\|u\right\|_{p(t)}^{p(t)}$ on both sides of the above equality and integrating by parts gives
	\begin{align*}
		&p^2\left\|u\right\|_{p(t)}^{p(t)}\partial_t\left\|u\right\|_{p(t)}\\
		&=-p'\left\|u\right\|_{p(t)}^{p(t)+1}\ln\left(\left\|u\right\|_{p(t)}^{p(t)}\right)+pp'\left\|u\right\|_{p(t)}\int_{\Omega}u^p\ln udv\\
		&\quad-p^2(p-1)\left\|u\right\|_{p(t)}\int_{\Omega}u^{p-2}\left| \nabla u\right|^2dv-p^2\left\|u\right\|_{p(t)}\int_{\Omega}\phi u^pdv.
	\end{align*}
	Dividing by $\left\|u\right\|_{p(t)}^{p(t)+1}$ yields
	\begin{align*}
		&p^2\partial_t\ln \left\|u\right\|_{p(t)}\\
		&=-p'\ln \left\|u\right\|_{p(t)}^{p(t)}+\frac{pp'}{\left\|u\right\|_{p(t)}^{p(t)}}\int_{\Omega}u^p\ln udv-\frac{4(p-1)}{\left\|u\right\|_{p(t)}^{p(t)}}\int_{\Omega}\left|\nabla u^{\frac{p}{2}}\right|^2dv-\frac{p^2}{\left\|u\right\|_{p(t)}^{p(t)}}\int_{\Omega}\phi u^pdv.
	\end{align*}
	Let 
	$$w=\frac{u^{\frac{p}{2}}}{\left\|u^{\frac{p}{2}}\right\|_2},$$
	then $\left\|w\right\|_2=1.$
	Therefore
	\begin{align}\label{2.8*}
		p^2\partial_t\ln \left\|u\right\|_{p(t)}=p'\left[\int_{\Omega}w^2\ln w^2dv-\frac{4(p-1)}{p'}\int_{\Omega}\left|\nabla w\right|^2dv-\frac{p^2}{p'}\int_{\Omega}\phi w^2dv\right].
	\end{align}
	Recall $\phi=\frac{1}{4}(f+R-n)\geq\frac{1}{4}(R-n)$,
	so
	\begin{align*}
		p^2\partial_t\ln \left\|u\right\|_{p(t)}\leq p'\left[\int_{\Omega}w^2\ln w^2dv-\frac{4(p-1)}{p'}\int_{\Omega}\left|\nabla w\right|^2dv-\frac{p^2}{4p'}\int_{\Omega}(R-n)w^2dv\right].
	\end{align*}
	Now choosing $\tau=\frac{p-1}{p'}$ in the optimal Log-Sobolev inequality on Ricci shrinkers proved by \cite[Theorem 1]{LW} gives
	$$\int_{\Omega}w^2\ln w^2dv\leq\frac{p-1}{p'}\int_{\Omega}\left(4\left|\nabla w\right|^2+Rw^2\right)dv-\left[\boldsymbol{\mu}+n+\frac{n}{2}\ln(4\pi\tau)\right].$$
	Plugging this inequality into (\ref{2.8*}) and noticing $p\geq1$, we arrive at
	\begin{align*}
		\partial_t\ln \left\|u\right\|_{p(t)}&\leq \frac{p'}{p^2}\left( \frac{p-1-\frac{p^2}{4}}{p'}\int_{\Omega}Rw^2dv-\left[\boldsymbol{\mu}+n+\frac{n}{2}\ln(4\pi\tau)\right]\right)+\frac{n}{4}\\
		&\leq -\frac{p'}{p^2}\left[\boldsymbol{\mu}+n+\frac{n}{2}\ln(4\pi\tau)\right]+\frac{n}{4}\\
		&=-\frac{1}{T}\left[\boldsymbol{\mu}+n+\frac{n}{2}\ln\frac{4\pi t(T-t)}{T}\right]+\frac{n}{4}.
	\end{align*}
	Integrating this inequality from $0$ to $T$ implies
	$$\left\|u(x,T)\right\|_\infty\leq \frac{e^{-\boldsymbol{\mu}+\frac{nT}{4}}}{(4\pi T)^\frac{n}{2}}\left\|u(x,0)\right\|_1.$$
	Since $u(y,0)$ is an arbitrary non-negative function with compact support in $\Omega$, $\Omega$ is arbitrary and 
	$$u(x,T)=\int_{\Omega}u(y,0)H_\phi^{\Omega}(x,y,T)dv(y),$$
	hence we conclude
	\begin{equation}\label{2.9*}
		H_\phi^{\Omega}(x,y,T)\leq\frac{e^{-\boldsymbol{\mu}+\frac{nT}{4}}}{(4\pi T)^\frac{n}{2}}.
	\end{equation}
	The upper bound of (\ref{2.9*}) behaves worse and worse as $t$ increases, but due to the semigroup property of $H_\phi^{\Omega}(x,y,t)$, we know 
	$H_\phi^{\Omega}(x,y,t)\leq\sqrt{H_\phi^{\Omega}(x,x,t)H_\phi^{\Omega}(y,y,t)}$ and moreover,  by (\ref{1.3*}),
	\begin{align}
		\frac{d}{dt}H_\phi^{\Omega}(x,x,2t)= \frac{d}{dt}\int_{\Omega}\left(H_\phi^{\Omega}(x,y,t)\right)^2dv&=-2\int_{\Omega}H_\phi^{\Omega}(x,y,t)\left(-\Delta+\phi\right)H_\phi^{\Omega}(x,y,t)dv\nonumber\\
		&=-2\int_{\Omega}\left(\left|\nabla H_\phi^{\Omega}(x,y,t)\right|^2+\phi\left(H_\phi^{\Omega}(x,y,t)\right)^2\right)dv \nonumber\\
		&\leq0.\label{2.5*}
	\end{align}
	hence for $t>1$, we have
	\begin{equation}\label{2.10*}
		H_\phi^\Omega(x,y,t)\leq(4\pi )^{-\frac{n}{2}}e^{-\boldsymbol{\mu}+\frac{n}{4}}.
	\end{equation}
	Combining (\ref{2.9*}) and (\ref{2.10*}) completes the proof.
\end{proof}

Utilizing the above upper bound of $H_\phi(x,y,t)$ and the iterative techniques of \cite{AG}, we are ready to give the Gaussian upper bound of $H_\phi(x,y,t)$ which is effective for short time. On the basis of this short time Gaussian upper bound, we are able to obtain a long time upper bound, and this will be proved later. Now we concentrate on the short time estimates.

As before, we adopt an approximation argument. Let $\Omega$ be any fixed compact set on $(M^n,g,f)$, $K\subset\subset \Omega$ be compact and $u(x,t)\in W^2_0(\Omega)\bigcap C^\infty(\Omega)$ be a solution to 
\begin{equation}\label{2.11*}
	\partial_t u=(\Delta-\phi)u
\end{equation}
with an initial condition having a support in $K$. For $D>2$, define
$$E_D(t)=\int_{\Omega}u^2(x,t)\exp\left(\frac{d^2\left(x,K\right)}{Dt}\right)dv.$$
\begin{proposition}\label{p2.2*}
	For any $\beta>1$, $D>2$ and $T>t>0$, suppose for all $0<s< T$ that  $I(s)=\int_{\Omega}u^2(\cdot,s)dv\leq\gamma(2s)$, then there exists some $\theta=\theta(\beta, D,n)$ such that
	\begin{equation}\label{2.12*}
		E_D(t)\leq \frac{\gamma(t)}{\theta^{\frac{n}{2}}},
	\end{equation}
	where $\gamma(t)$ is the same as in (\ref{2.6*}).
\end{proposition}
\begin{proof}
	Given $r>0$ and $T>t>0$, set
	$$I_r(t)=\int_{\Omega\setminus B(K,r)}u^2(x,t)dv,$$
	where $B(K,r)$ is the $r$-neighborhood of $K$. To estimate $I_r(t)$, we consider
	$$\int_{\Omega}u^2(x,t)e^{\xi(x,t)}dv,$$
	where
	$$\xi(x,t)=-\frac{d^2(x)}{2(T-t)},\quad d(x)=
	\left\{\
	\begin{array}{lr}
		r-d(x,K) ,&x\in B(K,r),\\
		0 ,&\mbox{otherwise}.
	\end{array}
	\right.$$
	Direct computations yield
	\begin{align}
		\frac{d}{dt}\int_{\Omega}u^2(x,t)e^{\xi(x,t)}dv&=\int_{\Omega}u^2e^{\xi}\xi_tdv+\int_{\Omega}2uu_te^{\xi}dv\nonumber\\
		&\leq-\frac{1}{2}\int_{\Omega}u^2e^{\xi}\left| \nabla \xi \right|^2dv+2\int_{\Omega}u(\Delta u-\phi u)e^{\xi}dv.\label{2.13*}
	\end{align}
	In view of (\ref{1.3*}), 
	$$\int_{\Omega}\left(\left|\nabla(ue^{\frac{\xi}{2}})\right|^2+\phi u^2e^\xi\right)dv\geq0,$$
	therefore
	\begin{align}
		&2\int_{\Omega}u(\Delta u-\phi u)e^{\xi}dv\nonumber\\
		&\leq 2\int_{\Omega}ue^\xi\Delta udv+2\int_{\Omega}\left|\nabla(ue^{\frac{\xi}{2}})\right|^2dv\nonumber\\
		&=2\int_{\Omega}ue^\xi\Delta udv-2\int_{\Omega}ue^{\frac{\xi}{2}}\Delta(ue^{\frac{\xi}{2}})dv\nonumber\\
		&=2\int_{\Omega}ue^\xi\Delta udv-2\int_{\Omega}ue^{\frac{\xi}{2}}\left[u\left(\frac{1}{2}e^{\frac{\xi}{2}}\Delta\xi+\frac{1}{4}e^{\frac{\xi}{2}}\left| \nabla \xi \right|^2\right)+e^{\frac{\xi}{2}}\Delta u+2\left\langle\nabla u,\nabla e^{\frac{\xi}{2}} \right\rangle\right]dv\nonumber\\
		&=-\int_{\Omega}u^2e^\xi\Delta\xi dv-\frac{1}{2}\int_{\Omega}u^2e^\xi\left| \nabla \xi \right|^2dv-2\int_{\Omega}ue^{\xi}\left\langle\nabla u,\nabla \xi \right\rangle dv\nonumber\\
		&=\int_{\Omega}\left\langle\nabla(u^2e^\xi),\nabla\xi\right\rangle dv-\frac{1}{2}\int_{\Omega}u^2e^\xi\left| \nabla \xi \right|^2dv-2\int_{\Omega}ue^{\xi}\left\langle\nabla u,\nabla \xi \right\rangle dv\nonumber\\
		&=\frac{1}{2}\int_{\Omega}u^2e^{\xi}\left| \nabla \xi \right|^2dv.\label{2.14*}
	\end{align}
	Combining (\ref{2.13*}) and (\ref{2.14*}),
	\begin{equation*}
		\frac{d}{dt}\int_{\Omega}u^2(x,t)e^{\xi(x,t)}dv\leq0.
	\end{equation*}
	In other words, for $0<s<t<T$,
	$$\int_{\Omega}u^2(x,t)e^{\xi(x,t)}dv\leq\int_{\Omega}u^2(x,s)e^{\xi(x,s)}dv.$$
	As a consequence, for $0<\rho<r$,
	\begin{align*}
		I_r(t)&\leq\int_{\Omega}u^2(x,t)e^{\xi(x,t)}dv\\
		&\leq \int_{\Omega}u^2(x,s)e^{\xi(x,s)}dv\\
		&\leq \int_{\Omega\setminus B(K,\rho)}u^2(x,s)dv+\int_{B(K,\rho)}u^2(x,s)e^{-\frac{(r-\rho)^2}{2(T-s)}}dv\\
		&\leq I_\rho(s)+\exp\left(-\frac{(r-\rho)^2}{2(T-s)}\right)\int_{B(K,\rho)}u^2(x,s)dv.
	\end{align*}
	Setting $T\longrightarrow t$, we have
	\begin{equation}\label{2.15*}
		I_r(t)\leq I_\rho(s)+\exp\left(-\frac{(r-\rho)^2}{2(t-s)}\right)\int_{B(K,\rho)}u^2(x,s)dv.
	\end{equation}
	Once obtaining the inequality (\ref{2.15*}), we can iterate step by step as in \cite{AG}(see also \cite[section 4]{Wu} for details) to show that (\ref{2.12*}) is controlled by the following $I(t)$ without further requirements. Briefly, once (\ref{2.15*})) is established, the remainder of the proof is the same as the proof of Wu\cite[Lemma 4.2 and Proposition 4.1]{Wu}. Indeed, using
	\begin{equation*}
		\int_{B(K,\rho)}u^2(\cdot,s)\,dv \le I(s)
	\end{equation*}
	and \(I(s)\le\gamma(2s)\), we obtain the recursive inequality corresponding to (4.2) of \cite{Wu}(4.2). The function \(\gamma\) satisfies
	\begin{equation*}
		\gamma(2s)\le C(n)\left(\frac{t}{s}\right)^{n/2}\gamma(t),\quad 0<s\le t,
	\end{equation*}
	and hence Wu's choices
	\begin{equation*}
		r_k=\left(\frac12+\frac{1}{k+2}\right)r,\quad t_k=\frac{t}{\beta^k}
	\end{equation*}
	give that for $D_0=D_0(\beta)$,
	\begin{equation*}
		I_r(t)\le C(n,\beta)\gamma(t)\exp\left(-\frac{r^2}{D_0(\beta)t}\right).
	\end{equation*}
	The dyadic annulus argument in \cite[Proposition 4.1, step one]{Wu} then yields the desired estimate for \(D\ge 5D_0\). For \(2<D<5D_0\), step two therein of the same proof applies due to the nonnegativity of the quadratic form (\ref{1.3*}) implies the required monotonicity of
	\begin{equation*}
		\int_\Omega u^2(x,t)\exp\Bigl(\frac{d^2(x,K)}{2(t+s)}\Bigr)dv.
	\end{equation*}
	This completes the proof.
\end{proof}

Now the Gaussian upper bound follows by this proposition immediately.
\begin{theorem}\label{t2.3}
	For any Ricci shrinker $(M^n,g,f)$, $x,y\in M^n,t>0$ and $\varepsilon>0$, there exists a constant $\theta=\theta(n,\varepsilon)>0$ such that
	\begin{equation}\label{2.16*}
		H_\phi(x,y,t)\leq\frac{4}{\theta^{\frac{n}{2}}}\gamma(t)\exp\left(-\frac{d^2(x,y)}{(4+\varepsilon)t}\right),
	\end{equation}
	where $\gamma(t)$ is the same as in (\ref{2.6*}).
\end{theorem}
\begin{proof}
	By semigroup property, 
	$$H_\phi^\Omega(x,y,t)=\int_\Omega H_\phi^\Omega\left(x,z,\frac{t}{2}\right)H_\phi^\Omega\left(z,y,\frac{t}{2}\right)dv(z),$$
	then by the triangle inequality $d^2(x,y)\leq 2\left(d^2(x,z)+d^2(z,y)\right)$, for all $\Omega$ large enough,
	\begin{align*}
		H_\phi^\Omega(x,y,t)&\leq \int_\Omega H_\phi^\Omega\left(x,z,\frac{t}{2}\right)e^{\frac{d^2(x,z)}{Dt}}H_\phi^\Omega\left(z,y,\frac{t}{2}\right)e^{\frac{d^2(y,z)}{Dt}}e^{-\frac{d^2(x,y)}{2Dt}}dv(z)\\
		&\leq e^{-\frac{d^2(x,y)}{2Dt}}\left[ \int_\Omega\left( H_\phi^\Omega\left(x,z,\frac{t}{2}\right)e^{\frac{d^2(x,z)}{Dt}}\right)^2 \right]^{\frac{1}{2}}\left[\int_\Omega  \left(H_\phi^\Omega\left(z,y,\frac{t}{2}\right)e^{\frac{d^2(z,y)}{Dt}}\right)^2 \right]^{\frac{1}{2}}\\
		&\leq E_D^{\frac{1}{2}}\left(x,\frac{t}{2}\right)E_D^{\frac{1}{2}}\left(y,\frac{t}{2}\right)e^{-\frac{d^2(x,y)}{2Dt}}.
	\end{align*}
	Finally, combining Proposition \ref{p2.2*} with $\beta=2, D=2+\varepsilon/2$ and an approximation argument completes the proof.
\end{proof}

\section{Upper bounds: Part A}\label{sec-upper}
In this section, we are going to establish the potential estimates of heat kernel upper bounds. To this end, first we need the following fundamental differential equation.

\begin{lemma}\label{l2.1}
	Let $P(x,y,t)=H_{\phi}(x,y,t)\exp\left(\frac{f(x)}{2}\tanh\frac{t}{2}\right)$, then
	\begin{align}
		\left(\partial_t-\Delta_x\right) P\left( x,y,t\right)+\tanh \frac{t}{2} \left<\nabla_x P,\nabla_x f(x) \right>=\frac{nP}{4}\left( 1-\tanh\frac{t}{2}\right)-\frac{1}{4}\left( \tanh \frac{t}{2}-1\right) ^{2}PR(x).\label{3.1}
	\end{align}
\end{lemma}
\begin{proof}
	Direct computations give
	\begin{align*}
		\partial_tP&=\left( \partial_t H_{\phi}\right) \exp \left( \frac{f(x)}{2}\tanh \frac{t}{2}\right) +H_{\phi}\exp \left( \frac{f(x)}{2}\tanh \frac{t}{2}\right)\frac{f(x)}{4}\left( \cosh ^{-2}\frac{t}{2}\right)
	\end{align*}
	and
	\begin{align*}
		&\Delta_xP\\
		&=\left(\Delta_x H_{\phi}\right) \exp \left( \frac{f(x)}{2}\tanh\frac{t}{2}\right) +2\left<\nabla_x H_{\phi},\nabla_x \exp \left(\frac{f(x)}{2} \tanh\frac{t}{2}\right) \right>+H_{\phi} \Delta _{x}\exp \left(\frac{f(x)}{2} \tanh\frac{t}{2}\right) \\
		&=\left(\Delta_xH_{\phi}\right) \exp \left( \frac{f(x)}{2}\tanh\frac{t}{2}\right) +2\left<\nabla_x H_{\phi},\exp \left(\frac{f(x)}{2} \tanh\frac{t}{2}\right) \dfrac{\tanh \frac{t}{2}}{2}\nabla_x f(x)\right>\\
		&\quad+H_{\phi}\left[\exp \left(\frac{f(x)}{2} \tanh\frac{t}{2}\right)\dfrac{\tanh \frac{t}{2}}{2}\Delta _{x}f(x)+\exp \left( \frac{f(x)}{2}\tanh \frac{t}{2}\right) \frac{\tanh ^{2}\frac{t}{2}}{4}\left|\nabla_xf(x)\right|^2\right].
	\end{align*}
	Note that $\left(\partial_t-\Delta_x\right)H_{\phi}=-\frac{1}{4}\left(f(x)+R(x)-n\right)H_{\phi}$,
	therefore
	\begin{align*}
		&\left(\partial_t-\Delta_x\right) P\\
		&=-\frac{P}{4}\left(f(x)+R(x)-n\right)+\frac{f(x)}{4}\left( \cosh ^{-2}\frac{t}{2}\right)P\\
		&\quad-\tanh \frac{t}{2}\exp \left(\frac{f(x)}{2} \tanh\frac{t}{2}\right)\left<\nabla_x H_{\phi},\nabla_x f(x)\right>-P\left[\frac{\tanh \frac{t}{2}}{2}\Delta _{x}f\left( x\right)+\frac{\tanh ^{2}\frac{t}{2}}{4}\left|\nabla_xf(x)\right|^2\right].
	\end{align*}
	On the other hand,
	\begin{align*}
		\left\langle\nabla_x P,\nabla_x f(x)\right\rangle&=\left\langle\nabla_x \left(H_{\phi}\exp\left(\frac{f(x)}{2}\tanh\frac{t}{2}\right)\right),\nabla_x f(x)\right\rangle\\
		&=\exp\left(\frac{f(x)}{2}\tanh\frac{t}{2}\right)\left\langle \nabla_x H_\phi, \nabla_x f(x)\right\rangle+\frac{P\tanh \frac{t}{2}}{2}\left|\nabla_x f(x)\right|^2.
	\end{align*}
	Summarily,
	\begin{align*}
		&\left(\partial_t-\Delta_x\right) P\\
		&=-\frac{P}{4}\left(f(x)+R(x)-n\right)+\frac{f(x)}{4}\left( \cosh ^{-2}\frac{t}{2}\right)P-\tanh \frac{t}{2}\left<\nabla_x P,\nabla_x f(x)\right>+\frac{P\tanh^2 \frac{t}{2}}{2}\left|\nabla_x f(x)\right|^2\\
		&\quad-P\left[\frac{\tanh\frac{t}{2}}{2}\Delta _{x}f\left( x\right)+\frac{\tanh ^{2}\frac{t}{2}}{4}\left|\nabla_xf(x)\right|^2\right].
	\end{align*}
	In light of (\ref{1.1}) and the identity $\frac{1}{\cosh^2 \theta}=1-\tanh^2 \theta$,  we have 
	\begin{align*}
		\left(\partial_t-\Delta_x\right) P\left( x,y,t\right)+\tanh \frac{t}{2} \left<\nabla_x P,\nabla_x f(x) \right>=\dfrac{nP}{4}\left( 1-\tanh\frac{t}{2}\right)-\frac{1}{4}\left(1-\tanh \frac{t}{2}\right) ^{2}PR(x).
	\end{align*}
\end{proof}

By similar proof of  Theorem \ref{t2.1}, we have
\begin{proposition}\label{p7.1}
	For any Ricci shrinker $(M^n,g,f)$ and $x,y\in M^n,t\in\left(0,1\right]$,
	\begin{equation*}
		H_\phi(x,y,t)e^{\frac{f(x)}{2}\tanh\frac{t}{2}}\leq 
		(4\pi t)^{-\frac{n}{2}}e^{-\boldsymbol{\mu}+\frac{n}{2}t}.
	\end{equation*}
\end{proposition}
\begin{proof}
	As before, we only need to consider the corresponding heat kernel $H^\Omega_{\phi}$ on compact domains $\Omega$ with smooth boundaries and approximate $H_\phi$ by $H^\Omega_\phi$. Following the same argument as in Lemma \ref{l2.1}, let $P(x,y,t)=H^\Omega_{\phi}(x,y,t)\exp\left(\frac{f(x)}{2}\tanh\frac{t}{2}\right)$, then
	\begin{align}
		\left(\partial_t-\Delta_x\right) P\left( x,y,t\right)+\tanh \frac{t}{2} \left<\nabla_x P,\nabla_x f(x) \right>=\frac{nP}{4}\left( 1-\tanh  \frac{t}{2}\right)-\frac{1}{4}\left( \tanh \frac{t}{2}-1\right) ^{2}PR(x).\label{7.1}
	\end{align}
	Let $P(x,z,t)=H^\Omega_{\phi}(x,z,t)\exp\left(\frac{f(x)}{2}\tanh\frac{t}{2}\right)$ and $h(z)$ be  any non-negative smooth function with compact support in $\Omega$. Then for any $0<T\leq1$, $t\in\left[0,T\right]$ and by (\ref{7.1}), 
	$$u(x,t)=\int_{\Omega}h(z)P(x,z,t)dv(z)\in C^\infty(\Omega)\bigcap W^2_0(\Omega)$$
	satisfies
	\begin{align}\label{7.2}
		\left(\partial_t-\Delta_x\right)u+\tanh\frac{t}{2}\left<\nabla_x u,\nabla_x f(x)\right>=\frac{nu}{4}\left(1-\tanh\frac{t}{2}\right)-\frac{1}{4}\left( \tanh \frac{t}{2}-1\right)^{2}uR(x).
	\end{align}
	Set 
	$$\left\|u\right\|_{p(t)}=\left(\int_{\Omega}\left|u\right|^{p(t)}dv(x)\right)^{\frac{1}{p(t)}},$$
	where $p(t)=\frac{T}{T-t}$, $t\in\left[0,T\right]$, which implies $p(0)=1$ and $p(T)=\infty$.
	By (\ref{7.2}), we compute
	\begin{align*}
		\partial_t\left\|u\right\|_{p(t)}
		&=-\frac{p'}{p^2}\left\|u\right\|_{p(t)}\ln\left(\left\|u\right\|_{p(t)}^{p(t)}\right)+\frac{\left\|u\right\|_{p(t)}^{1-p(t)}}{p(t)}\left[p'\int_{\Omega}u^p\ln udv(x)+p\int_{\Omega}u^{p-1}\partial_t udv(x)\right]\\
		&\leq-\frac{p'}{p^2}\left\|u\right\|_{p(t)}\ln\left(\left\|u\right\|_{p(t)}^{p(t)}\right)+\frac{\left\|u\right\|_{p(t)}^{1-p(t)}}{p(t)}p'\int_{\Omega}u^p\ln udv(x)\\
		&\quad+\left\|u\right\|_{p(t)}^{1-p(t)}\int_{\Omega}u^{p-1}\left(\Delta_xu+\frac{nu}{4}\left(1-\tanh\frac{t}{2}\right)\right)dv(x)\\
		&\quad-\left\|u\right\|_{p(t)}^{1-p(t)}\int_{\Omega}u^{p-1}\left(\tanh\frac{t}{2}\left<\nabla_x u,\nabla_x f(x)\right>+\frac{\left( \tanh \frac{t}{2}-1\right) ^{2}uR(x)}{4}\right)dv(x).
	\end{align*}
	Multiplying by $p^2\left\|u\right\|_{p(t)}^{p(t)}$ on both sides of the above equality and integrating by parts gives
	\begin{align*}
		&p^2\left\|u\right\|_{p(t)}^{p(t)}\partial_t\left\|u\right\|_{p(t)}\\
		&\leq-p'\left\|u\right\|_{p(t)}^{p(t)+1}\ln\left(\left\|u\right\|_{p(t)}^{p(t)}\right)+pp'\left\|u\right\|_{p(t)}\int_{\Omega}u^p\ln udv(x)\\
		&\quad-p^2(p-1)\left\|u\right\|_{p(t)}\int_{\Omega}u^{p-2}\left|\nabla_x u\right|^2dv(x)+p\left\|u\right\|_{p(t)}\tanh \frac{t}{2}\int_{\Omega}u^p\Delta_x f(x)dv(x)\\
		&\quad+p^2\left\|u\right\|_{p(t)}\int_{\Omega}\left(\frac{n}{4}\left(1-\tanh\frac{t}{2}\right)-\frac{\left( \tanh \frac{t}{2}-1\right)^{2}R(x)}{4}\right)u^pdv(x).
	\end{align*}
	Since $\Delta f=\frac{n}{2}-R\leq\frac{n}{2}$, $0<\tanh\frac{t}{2}<1$ and $p\geq1$,
	\begin{align*}
		&p^2\left\|u\right\|_{p(t)}^{p(t)}\partial_t\left\|u\right\|_{p(t)}\\
		&\leq-p'\left\|u\right\|_{p(t)}^{p(t)+1}\ln\left(\left\|u\right\|_{p(t)}^{p(t)}\right)+pp'\left\|u\right\|_{p(t)}\int_{\Omega}u^p\ln udv(x)\\
		&\quad-p^2(p-1)\left\|u\right\|_{p(t)}\int_{\Omega}u^{p-2}\left|\nabla_x u\right|^2dv(x)+p\left\|u\right\|_{p(t)}\tanh \frac{t}{2}\int_{\Omega}u^p\left(\frac{n}{2}-R(x)\right)dv(x)\\
		&\quad+p^2\left\|u\right\|_{p(t)}\int_{\Omega}\left(\frac{n}{4}\left(1-\tanh\frac{t}{2}\right)-\frac{\left( \tanh \frac{t}{2}-1\right)^{2}R(x)}{4}\right)u^pdv(x)\\
		&\leq-p'\left\|u\right\|_{p(t)}^{p(t)+1}\ln\left(\left\|u\right\|_{p(t)}^{p(t)}\right)+pp'\left\|u\right\|_{p(t)}\int_{\Omega}u^p\ln udv(x)\\
		&\quad-p^2(p-1)\left\|u\right\|_{p(t)}\int_{\Omega}u^{p-2}\left|\nabla_x u\right|^2dv(x)\\
		&\quad+p^2\left\|u\right\|_{p(t)}\int_{\Omega}\left(\frac{n}{4}\left(1+\tanh\frac{t}{2}\right)-\left( \frac{\left( \tanh \frac{t}{2}-1\right)^{2}}{4}+\frac{\tanh\frac{t}{2}}{p}\right)R(x)\right)u^pdv.
	\end{align*}
	
	Dividing by $\left\|u\right\|_{p(t)}^{p(t)+1}$ yields
	\begin{align*}
		&p^2\partial_t\ln \left\|u\right\|_{p(t)}\\
		&\leq-p'\ln \left\|u\right\|_{p(t)}^{p(t)}+\frac{pp'}{\left\|u\right\|_{p(t)}^{p(t)}}\int_{\Omega}u^p\ln udv-\frac{4(p-1)}{\left\|u\right\|_{p(t)}^{p(t)}}\int_{\Omega}\left|\nabla_xu^{\frac{p}{2}}\right|^2dv(x)\\
		&\quad+\frac{p^2}{\left\|u\right\|_{p(t)}^{p(t)}}\int_{\Omega}\left(\frac{n}{4}\left(1+\tanh\frac{t}{2}\right)-\left(\frac{\left( \tanh \frac{t}{2}-1\right)^{2}}{4}+\frac{\tanh\frac{t}{2}}{p}\right)R(x)\right)u^pdv(x).
	\end{align*}
	Let 
	$$w=\frac{u^{\frac{p}{2}}}{\left\|u^{\frac{p}{2}}\right\|_2},$$
	then $\left\|w\right\|_2=1$.
	Consequently,
	\begin{align}\label{7.3}
		p^2\partial_t\ln \left\|u\right\|_{p(t)}&\leq p'\left[\int_{\Omega}w^2\ln w^2dv(x)-\frac{4(p-1)}{p'}\int_{\Omega}\left|\nabla_x w\right|^2dv(x)\right]\nonumber\\
		&\quad+p'\left[\frac{p^2}{4p'}\int_{\Omega}\left(n\left(1+\tanh\frac{t}{2}\right)-\left(\left( \tanh \frac{t}{2}-1\right)^{2}+\frac{4\tanh\frac{t}{2}}{p}\right)R(x)\right)w^2dv(x)\right].
	\end{align}
	
	Now choosing $\tau=\frac{(p-1)}{p'}$ in the optimal Log-Sobolev inequality on Ricci shrinkers proved by \cite[Theorem 1]{LW} gives
	$$\int_{\Omega}w^2\ln w^2dv(x)\leq\frac{p-1}{p'}\int_{\Omega}\left(4\left|\nabla_x w\right|^2+Rw^2\right)dv(x)-\left[\boldsymbol{\mu}+n+\frac{n}{2}\ln(4\pi\tau)\right].$$
	Plugging this inequality into (\ref{7.3}) and noticing $p\geq1$, we arrive at
	\begin{align*}
		&\partial_t\ln \left\|u\right\|_{p(t)}\\
		&\leq \frac{p'}{p^2}\left( \frac{p-1-\frac{\left( \tanh \frac{t}{2}-1\right)^{2}p^2}{4}-p\tanh\frac{t}{2}}{p'}\int_{\Omega}R(x)w^2dv-\left[\boldsymbol{\mu}+n+\frac{n}{2}\ln(4\pi\tau)\right]\right)+\frac{n}{4}\left(1+\tanh\frac{t}{2}\right)\\
		&=\frac{p'}{p^2}\left( \frac{-\left(\frac{p\left(1-\tanh\frac{t}{2}\right)}{2}-1 \right)^2}{p'}\int_{\Omega}R(x)w^2dv(x)-\left[\boldsymbol{\mu}+n+\frac{n}{2}\ln(4\pi\tau)\right]\right)+\frac{n}{4}\left(1+\tanh\frac{t}{2}\right)\\
		&\leq -\frac{p'}{p^2}\left[\boldsymbol{\mu}+n+\frac{n}{2}\ln(4\pi\tau)\right]+\frac{n}{4}\left(1+\tanh\frac{t}{2}\right)\\
		&=-\frac{1}{T}\left[\boldsymbol{\mu}+n+\frac{n}{2}\ln\frac{4\pi t(T-t)}{T}\right]+\frac{n}{4}\left(1+\tanh\frac{t}{2}\right)\\
		&\leq-\frac{1}{T}\left[\boldsymbol{\mu}+n+\frac{n}{2}\ln\frac{4\pi t(T-t)}{T}\right]+\frac{n}{2}.
	\end{align*}
	Integrating this inequality from $0$ to $T\in\left(0,1\right]$ implies
	\begin{align*}
		\left\|u(x,T)\right\|_\infty \leq\frac{\exp\left(-\boldsymbol{\mu}+\frac{nT}{2}\right)}{(4\pi T)^\frac{n}{2}}\left\|u(x,0)\right\|_1.
	\end{align*}
	Since $h$ is an arbitrary non-negative function with compact support in $\Omega$, $\Omega$ is arbitrary and 
	$$u(x,T)=\int_{\Omega}h(z)P(x,z,T)dv(z),$$
	hence we conclude
	\begin{equation*}
		H_\phi^{\Omega}(x,y,T)\exp\left(\frac{f(x)}{2}\tanh\frac{T}{2}\right)\leq\frac{e^{-\boldsymbol{\mu}+\frac{nT}{2}}}{(4\pi T)^\frac{n}{2}}.
	\end{equation*}
\end{proof}

It's natural to extend the above short time estimate to long time via the following maximum principle.
\begin{lemma}\label{l7.2}
	For any Ricci shrinker $(M^n,g,f)$, $\delta>0$ and $t\in[t_0,t_1]\subset\left(0,+\infty\right)$, let $q(x,t)\leq C e^{(1-\varepsilon)f(x)\tanh\delta t}$ satisfy 
	\begin{equation*}
		Lq=\partial_tq-\Delta q+\tanh\delta t\left\langle\nabla q, \nabla f\right\rangle\leq0,
	\end{equation*}
	where $C>0$ and $0<\varepsilon\ll1$ are some uniform constants. If $q(x,t_0)\leq C_1$ on $M^n$, then $q(x,t)\leq C_1$ for all $x\in M^n$ and $t\in [t_0,t_1]$.
\end{lemma}
\begin{proof}
	Let $$\varphi(x,t)=e^{at+f\tanh\delta t},$$
	where $a>0$ is some constant to be determined. Direct computations yield
	\begin{align}
		L\varphi&=\varphi\left[a+\delta f\cosh^{-2}\delta t-\Delta f\tanh\delta t-\left|\nabla f\right|^2\tanh^2\delta t+\left\langle\nabla f, \nabla f\right\rangle\tanh^2\delta t\right]\nonumber\\
		&\geq\varphi\left[a-\frac{n}{2}\tanh\delta t\right],\label{7.4}
	\end{align}
	where we used the fact $\Delta f=\frac{n}{2}-R\leq\frac{n}{2}$.
	Since $0<\tanh\delta t<1$, hence if we choose $a\geq\frac{n}{2}$, then by (\ref{7.4}),
	\begin{equation*}
		L\varphi(x,t)\geq0.
	\end{equation*}
	Let $v(x,t)=q(x,t)-C_1$, then $Lv\leq0$ and $v(x,t_0)\leq0$ for all $x\in M^n$. Define $\bar{v}=e^{-At}v$ for some constant $A>0$. To complete the proof, we only need to verify that for any $\epsilon>0$, it holds that $\bar{v}\leq\epsilon\varphi$ on $M\times\left[t_0,t_1\right]$. If the statement were false, then there exists a spacetime point $(x',t')\in M\times\left[t_0,t_1\right]$ and some $\epsilon>0$ such that $\left(\bar{v}-\epsilon\varphi\right)(x',t')>0$.
	Since $q(x,t)\leq Ce^{(1-\varepsilon)f(x)\tanh\delta t}$ and $\tanh\delta t\geq\tanh\delta t_0$, we know $\left(\bar{v}-\epsilon\varphi\right)(x,t)\longrightarrow-\infty$ as $d(x,p)\longrightarrow+\infty$ uniformly in $t$, i.e. $\bar{v}-\epsilon\varphi<0$ for $d(x,p)$ large enough independent of $t$. Besides, $\left(\bar{v}-\epsilon\varphi\right)(x,t_0)<0$ for all $x\in M^n$. Consequently, there exists a $t''\in(t_0,t')$ and $x''\in M$ such that $\left(\bar{v}-\epsilon\varphi\right)(x,t)\leq0$ for all $(x,t)\in M^n\times[t_0,t'']$ and $\left(\bar{v}-\epsilon\varphi\right)(x'',t'')=0$. In other words, $(x'',t'')$ is a local maximum point, hence at $(x'',t'')$, 
	\begin{equation*}
		\partial_t\left(\bar{v}-\epsilon\varphi\right)\geq0,\quad \nabla\left(\bar{v}-\epsilon\varphi\right)=0,\quad \Delta\left(\bar{v}-\epsilon\varphi\right)\leq0,\quad \left(\bar{v}-\epsilon\varphi\right)(x'',t'')=0.
	\end{equation*}
	Therefore at $(x'',t'')$, 
	\begin{align*}
		0\leq L\left(\bar{v}-\epsilon\varphi\right)&=L\bar{v}-\epsilon L\varphi\\
		&\leq L\bar{v}\\
		&=-Ae^{-At}v+e^{-At}Lv\\
		&\leq -A\epsilon\varphi<0,
	\end{align*}
	which is a contradiction. Thus our proof is complete.
\end{proof}

With the help of this maximum principle, we conclude
\begin{theorem}\label{t7.3}
	For any Ricci shrinker $(M^n,g,f)$ and $x,y\in M^n,t>0$,
	\begin{equation*}
		H_\phi(x,y,t)e^{\frac{f(x)}{2}\tanh\frac{t}{2}}\leq 
		\gamma(t):=\left\{
		\begin{array}{lr}
			(4\pi t)^{-\frac{n}{2}}e^{-\boldsymbol{\mu}+\frac{n}{2}t} ,&t\leq1,\\
			C(n)(4\pi )^{-\frac{n}{2}}e^{-\boldsymbol{\mu}},&t>1.
		\end{array}
		\right.	
	\end{equation*}
\end{theorem}
\begin{proof}
	We have proved the case of $0<t\leq1$. For $t\geq1$ and fixed $y\in M^n$, let $q(x,t)=P(x,y,t)=H_{\phi}(x,y,t)\exp\left(\frac{f(x)}{2}\tanh\frac{t}{2}\right)$ for given $y\in M^n$. By Theorem \ref{t2.1}, we have $q(x,t)\leq (4\pi )^{-\frac{n}{2}}e^{-\boldsymbol{\mu}+\frac{n}{2}}e^{\frac{f(x)}{2}\tanh\frac{t}{2}}$ for $t\geq1$. 	Let $Q(x,y,t)=P(x,y,t)\exp\left(\int_{0}^{t}\frac{n}{4}\left(\tanh\frac{s}{2} -1\right)ds\right)$,	then by (\ref{7.1}),
	\begin{align*}
		&\left(\partial_t-\Delta_x\right) Q\left( x,y,t\right)+\tanh \frac{t}{2} \left<\nabla_x Q,\nabla_x f(x) \right>\leq 0.
	\end{align*}
	Applying the maximum principle of Lemma \ref{l7.2} to the above inequality implies $Q(x,y,t)\leq \sup\limits_{z\in M^n}Q(z,y,1)$ for all $x\in M^n$ and $t\geq1$, hence the proof is completed.
\end{proof}	

Besides point-wise estimates, we also need corresponding integral estimates as follows.
\begin{corollary}\label{c7.4}
	For any Ricci shrinker $(M^n,g,f)$,
	$$\int_{M^n}H_\phi(x,y,t)e^{\frac{f(x)}{2}\tanh\frac{t}{2}}dv(y)\leq \left(\frac{2}{1+e^{-t}}\right)^{n/2}.$$
	The equality holds true if and only if the Ricci shrinker is isometric to the Gaussian shrinker.
\end{corollary}	
\begin{proof}
	Let $dv_f:=e^{-f}dv$
	and denote by $	P_t^f=e^{t\Delta_f}$
	the \(f\) heat semigroup. Set $	q(t):=\tanh\frac t2$
	and define
	\begin{align*}
		I(x,t):=\int_{M^n}H_\phi(x,y,t)\exp\left(\frac{q(t)}2f(x)\right)dv(y).
	\end{align*}
	Recall that the Doob's transform gives
	\begin{equation*}
		H_\phi(x,y,t)=H_f(x,y,t)\exp\left(-\frac{f(x)+f(y)}2\right).
	\end{equation*}
	Consequently,
	\begin{align*}
		I(x,t)&=\exp\left(-\frac{1-q(t)}2f(x)\right)\int_{M^n}H_f(x,y,t)e^{-f(y)/2}dv(y)\\
		&=\exp\left(-\frac{1-q(t)}2f(x)\right)\int_{M^n}H_f(x,y,t)e^{f(y)/2}dv_f(y).
	\end{align*}
	Define $\lambda(t):=\frac{1-q(t)}2$.
	Since $	q(t)=\frac{e^t-1}{e^t+1}$, we have $\lambda(t)=\frac1{1+e^t}$.
	Thus
	\begin{align}
		I(x,t)=e^{-\lambda(t)f(x)}P_t^f\left(e^{f/2}\right)(x).\label{eq:IviaP}
	\end{align}
	It therefore remains to estimate $P_t^f\left(e^{f/2}\right)$. For \(s\ge0\), set $	\lambda(s):=\frac1{1+e^s}$, then $\lambda(0)=\frac12$ and
	\begin{align*}
		\lambda'(s)=-\frac{e^s}{(1+e^s)^2}=\lambda^2(s)-\lambda(s).
	\end{align*}
	Let $c(s):=\frac n2\int_0^s\lambda(\tau)d\tau$ and $\Phi(x,s):=\exp\left(\lambda(s)f(x)+c(s)\right).$
	Since $\Delta_f f=\frac n2-f$ and $\left|\nabla f\right|^2=f-R$, 
	we compute
	\begin{align*}
		\partial_s\Phi&=\Phi\left(\lambda' f+c'\right),\\
		\Delta_f\Phi&=\Phi\left(\lambda\Delta_f f+\lambda^2|\nabla f|^2\right)\\
		&=\Phi\left[\lambda\Bigl(\frac n2-f\Bigr)+\lambda^2(f-R)\right]\\
		&=\Phi\left[\frac n2\lambda+(\lambda^2-\lambda)f-\lambda^2R\right].
	\end{align*}
	Using \(\lambda'=\lambda^2-\lambda\) and \(c'=\frac n2\lambda\), we obtain the exact identity
	\begin{align}
		\left(\partial_s-\Delta_f\right)\Phi=\lambda^2R\Phi.\label{eq:super}
	\end{align}
	Since \(R\ge0\), we have
	\begin{align*}
		\left(\partial_s-\Delta_f\right)\Phi\ge0.
	\end{align*}
	Moreover, $	\Phi(x,0)=e^{f(x)/2}$.
	
	Since $0<\lambda(s)\le\frac12$, we have on each bounded time interval \(0\le s\le t\),
	\begin{align*}
		\Phi(\cdot,s)dv_f=e^{c(s)}e^{-(1-\lambda(s))f}dv\leq e^{c(t)}e^{-f/2}dv.
	\end{align*}
	Similarly, by \(0\le R\le f\), we have
	\begin{align*}
		R\Phi(\cdot,s)dv_f\leq e^{c(t)}f e^{-f/2}dv.
	\end{align*}
	Thus both right-hand sides are integrable due to the facts that $f$ has quadratic growth and $\left| B(p,r)\right|$ has at most polynomial growth. For details, see \cite{CZ}.
	
	Fix a nonnegative function \(\eta\in C_c^\infty(M)\) and define
	\begin{align*}
		J(s):=\int_{M^n}P_{t-s}^f\eta(y)\Phi(y,s)dv_f(y),\quad0\le s\le t.
	\end{align*}
	For \(0<s<t\),
	\begin{equation*}
		\frac{d}{ds}P_{t-s}^f\eta=-\Delta_fP_{t-s}^f\eta.
	\end{equation*}
	By the self-adjointness of $\Delta_f$ with respect to $dv_f$, we obtain
	\begin{align*}
		J'(s)&=\int_{M^n}\left[-\Delta_fP_{t-s}^f\eta\cdot\Phi+P_{t-s}^f\eta\cdot\partial_s\Phi\right]dv_f\\
		&=\int_{M^n}P_{t-s}^f\eta\left(\partial_s-\Delta_f\right)\Phi dv_f\\
		&=\int_{M^n}P_{t-s}^f\eta\lambda^2R\,\Phi dv_f\\
		&\ge0.
	\end{align*}
	Therefore \(J(0)\le J(t)\). Since \(\Phi(\cdot,0)=e^{f/2}\), then
	\begin{align*}
		J(0)=\int_{M^n}P_t^f\eta\,e^{f/2}dv_f.
	\end{align*}
	By symmetry of the \(f\)-heat semigroup,
	\begin{align*}
		J(0)=\int_{M^n}\eta P_t^f(e^{f/2})dv_f.
	\end{align*}
	On the other hand,
	\begin{align*}
		J(t)=\int_{M^n}\eta\,\Phi(\cdot,t)dv_f.
	\end{align*}
	It follows that for every nonnegative \(\eta\in C_c^\infty(M^n)\),
	\begin{equation*}
		\int_{M^n}\eta\left[P_t^f\left(e^{f/2}\right)-\Phi(\cdot,t)\right]dv_f\leq0.
	\end{equation*}
	Hence for any $x\in M^n$,
	\begin{align}
		P_t^f\left(e^{f/2}\right)(x)\leq\Phi(x,t)=e^{\lambda(t)f(x)+c(t)}\label{eq:semigroup}
	\end{align}
	Substituting \eqref{eq:semigroup} into \eqref{eq:IviaP} gives $	I(x,t)\leq e^{c(t)}$.
	
	It remains to compute \(c(t)\). By the definition,
	\begin{align*}
		\int_0^t\lambda(s)ds=\int_0^t\frac{ds}{1+e^s}=\ln\left(\frac{2}{1+e^{-t}}\right),
	\end{align*}
	which implies
	\begin{align*}
		c(t)=\frac n2\ln\left(\frac{2}{1+e^{-t}}\right)
	\end{align*}
	Therefore
	\begin{align*}
		I(x,t)\leq e^{c(t)}=\left(\frac{2}{1+e^{-t}}\right)^{n/2}.
	\end{align*}
	
	On \(n\)-dimensional Gaussian shrinker,
	\begin{align*}
		I(x,t)=\left(\frac{2}{1+e^{-t}}\right)^{n/2}.
	\end{align*}
	Moreover, applying Duhamel's principle to \eqref{eq:super} gives the exact defect identity
	\begin{align*}
		\left(\frac{2}{1+e^{-t}}\right)^{n/2}-I(x,t)=e^{-\lambda(t)f(x)}\int_0^tP_{t-s}^f\left(\lambda(s)^2Re^{\lambda(s)f+c(s)}\right)(x)ds.	
	\end{align*}
	Hence the equality holds if and only if $R=0$, i.e. it's isometric to the Gaussian shrinker.
\end{proof}
Note that due to the semigroup property of $H_\phi$, we have
\begin{align*}
	H_\phi(x,y,2t)e^{\frac{f(x)+f(y)}{2}\tanh\frac{t}{2}}=\int_{M^n}H_\phi(x,z,t)e^{\frac{f(x)}{2}\tanh\frac{t}{2}}H_\phi(z,y,t)e^{\frac{f(y)}{2}\tanh\frac{t}{2}}dv(z).
\end{align*}
Therefore combining Theorem \ref{t7.3} and Corollary \ref{c7.4} gives 
\begin{theorem}[\textbf{=Theorem \ref{t1.1'}}]\label{t3.8}
	For any Ricci shrinker and $t>0$,
	\begin{align*}
		H_\phi(x, y, t) \leq C(n)e^{-\boldsymbol{\mu}}\Gamma(t) \exp \left(-\frac{f(x)+f(y)}{2} \tanh \frac{t}{4}\right),
	\end{align*}
	where
	$\Gamma(t):=\left\{
	\begin{array}{lr}
		t^{-\frac{n}{2}} ,&t\leq1,\\
		1 ,&t>1.
	\end{array}
	\right.	$ 
	Hence by $H_{f}=H_{\phi} e^{\frac{f(x)+f(y)}{2}}$,
	\begin{align*}
		H_f(x, y, t) \leq C(n)e^{-\boldsymbol{\mu}}\Gamma(t) \exp \left(\frac{f(x)+f(y)}{2} \left(1-\tanh \frac{t}{4}\right)\right).
	\end{align*}
\end{theorem}

\begin{remark}\label{r7.5}
	Recall $H_\phi(x,y,t)$ on $\left(\mathbb{R}^1,g,e^{-\frac{x^2}{4}}dx\right)$, we know $\tanh\frac{t}{4}$ is sharp. Moreover, since 
	$$H^{\mathbb{R}^1}_\phi(x,y,t)e^{\frac{x^2}{8}\tanh\frac{t}{2}}=\frac{1}{\left(4\pi (1-e^{-t}) \right)^{\frac{1}{2}}}
	\exp \left[-\frac{\left((e^t+1)y-2e^{\frac{t}{2}}x\right)^2}{8(e^{2t}-1)}\right]$$
	and
	$$\int_{-\infty}^{\infty}H^{\mathbb{R}^1}_\phi(x,y,t)e^{\frac{x^2}{8}\tanh\frac{t}{2}}dy=\frac{\sqrt{2}}{\sqrt{1+e^{-t}}}\in\left(1,\sqrt{2}\right),$$
	hence the coefficient $e^{\frac{f(x)}{2}\tanh\frac{t}{2}}$ in Theorem \ref{t7.3} and Corollary \ref{c7.4} is sharp.
\end{remark}	

For $0<\beta<1$, write $H_\phi=H_\phi^\beta H_\phi^{1-\beta}$, then just combining Theorem \ref{t2.3} and \ref{t3.8}, we are able to obtain an upper bound including both $f(x)+f(y)$ and $d^2(x,y)$.
\begin{proposition}\label{p7.6}
	For any Ricci shrinker $(M^n,g,f)$, $x,y\in M^n,t>0$, $\varepsilon>0$ and $0<\beta<1$, there exists a constant $C=C(n,\varepsilon)$ such that
	\begin{equation*}
		H_\phi(x,y,t)\leq C\gamma(t)\exp\left[ -\beta\left( \frac{f(x)+f(y)}{2}\tanh\frac{t}{4}\right) -(1-\beta)\frac{d^2(x,y)}{(4+\varepsilon)t}\right],
	\end{equation*}
	where $\gamma(t)$ is the same as in (\ref{2.6*}).
\end{proposition}

Generally, it's hard to derive point-wise estimate of $H_\phi(x,y,t)e^{\frac{f(x)+f(y)}2\tanh\frac t4}$, but for its integral, we still have the following sharp estimate. Since its proof is almost the same as in Corollary \ref{c7.4}, we only sketch it.
\begin{proposition}
	For any Ricci shrinker $(M^n,g,f)$, $x\in M^n$ and $t>0$, 
	\begin{align*}
		\int_{M^n} H_\phi(x,y,t)e^{\frac{f(x)+f(y)}2\tanh\frac t4}dv(y)\leq e^{nt/4}.
	\end{align*}
	The equality holds true if and only if the Ricci shrinker is isometric to the Gaussian shrinker.
\end{proposition}

\begin{proof}
	Fix \(T>0\), and set $a:=\tanh\frac{T}{4},\alpha:=\frac{1-a}{2}$, and $\beta:=\frac{1+a}{2}.$
	By the Doob's transform,
	\begin{align}
		\int_{M^n} H_\phi(x,y,T)e^{F_T(x,y)}dv(y)=e^{-\alpha f(x)}P_T^f\!\left(e^{\beta f}\right)(x),
		\label{eq:4.11}
	\end{align}
	where \(F_T(x,y)=\frac a2(f(x)+f(y))\). This is the same semigroup quantity treated in the proof of Corollary \ref{c7.4}. Indeed, define
	\begin{align*}
		\lambda(s):=\frac{\beta}{\beta+(1-\beta)e^s},\quad c(s):=\frac n2\int_0^s\lambda(\tau)d\tau,\quad
		\Phi(x,s):=e^{\lambda(s)f(x)+c(s)}.
	\end{align*}
	Then
	\begin{align*}
		\lambda(0)=\beta,\quad\lambda(T)=\alpha,\quad\lambda'=\lambda^2-\lambda,
	\end{align*}
	and, exactly as in Corollary \ref{c7.4},
	\begin{align*}
		\left(\partial_s-\Delta_f\right)\Phi=\lambda(s)^2R\Phi\geq 0.
	\end{align*}
	The same heat semigroup duality argument gives
	\begin{align*}
		P_T^f\left(e^{\beta f}\right)(x)\leq\Phi(x,T)=e^{\alpha f(x)+c(T)}.
	\end{align*}
	Finally, similar computations implies $c(T)=\frac{nT}{4}$
	and hence \eqref{eq:4.11} yields
	\begin{align*}
		\int_{M^n}H_\phi(x,y,T)\exp\left[\frac{f(x)+f(y)}2\tanh\frac T4\right]dv(y)\leq e^{nT/4}.
	\end{align*}
	Moreover, the exact defect is
	\begin{align*}
		e^{\alpha f(x)+nT/4}-P_T^f\left(e^{\beta f}\right)(x)=\int_0^TP_{T-s}^f\left(\lambda(s)^2R\,e^{\lambda(s)f+c(s)}\right)(x)ds.
	\end{align*}
	Thus equality holds if and only if it's isometric to the Gaussian shrinker.
\end{proof}

\section{Upper bounds: Part B}\label{sec-upperb}
If we set one point in the heat kernel as the minimum point of potential function $f$, the estimates of Theorem \ref{t3.8} can be improved as follows.
\begin{proposition}\label{p3.10}
	For any Ricci shrinker $(M^n,g,f)$, $x\in M^n$ and $\delta\in\left(\frac{1}{2}, \frac{2}{3}\right)$, there exists a constant $C=C(n,\delta)$ such that for $t\geq \frac{2}{3\delta}$,
	\begin{equation}\label{2.17}
		H_\phi(p,x,t)\leq Ce^{\boldsymbol{-\mu}}e^{-\frac{f(x)}{2}(\tanh \delta t)^{-1}}.
	\end{equation}
\end{proposition}
To prove the above estimate, we need the following maximum principle whose proof is the same as that of Lemma \ref{l7.2} if we use the barrier function $\varphi(x,t)=e^{\frac{n}{2}t+f}$ which satisfies $L\varphi\geq0$.
\begin{lemma}\label{l2.5}
	For any Ricci shrinker $(M^n,g,f)$, $\delta\geq\frac{1}{2}$ and $t\in[t_0,t_1]\subset(0,+\infty)$, let $u(x,t)\leq C e^{\left(1-\varepsilon\right)f(x)}$ satisfy 
	\begin{equation}\label{3.11}
		Lu=\partial_tu-\Delta u+(\tanh\delta t)^{-1}\left\langle\nabla u, \nabla f\right\rangle\leq0,
	\end{equation}
	where $0<\varepsilon\ll 1$ and $C>0$ are some uniform constants. If $u(x,t_0)\leq C_1$ on $M^n$, then $u(x,t)\leq C_1$ for all $x\in M^n$ and $t\in [t_0,t_1]$.
\end{lemma}

Now we are ready to give the proof.
\begin{proof}[\textbf{Proof of Proposition \ref{p3.10}}]
	Let $u(x,t)=H_\phi(p,x,t)$, then $\partial_t u=(\Delta-\phi)u$. Therefore by similar computations as in Lemma \ref{l2.1} and recall $\phi=\frac{1}{4}(f+R-n)$, $R+\left|\nabla f\right|^2=f$ and $\Delta f=\frac{n}{2}-R$ , we compute
	\begin{align*}
		&(\partial_t-\Delta)\left(ue^{\frac{f}{2}(\tanh\delta t)^{-1}}\right)\\
		&= -ue^{\frac{f}{2}(\tanh\delta t)^{-1}}\left[\frac{1}{4}(1-2\delta)\left(1-(\tanh\delta t)^{-2}\right)f+\frac{1}{4}\left(1-(\tanh\delta t)^{-1}\right)^2R+\frac{n}{4}\left((\tanh\delta t)^{-1}-1\right)\right]\\
		&\quad-(\tanh\delta t)^{-1}\left\langle \nabla f,\nabla\left(ue^{\frac{f}{2}(\tanh\delta t)^{-1}}\right)\right\rangle\\
		&\leq -(\tanh\delta t)^{-1}\left\langle \nabla f,\nabla\left(ue^{\frac{f}{2}(\tanh\delta t)^{-1}}\right)\right\rangle.
	\end{align*}
	
	Set
	\begin{align*}
		A(t):=(\tanh \delta t)^{-1}\quad \mbox{and}\quad q(x,t):=u(x,t)e^{\frac12A(t)f(x)},
	\end{align*}
	then by the preceding computations, we have
	\begin{align*}
		\partial_t q-\Delta q+A(t)\langle\nabla q,\nabla f\rangle\le0.
	\end{align*}
	Now fix $\delta\in\left(\frac12,\frac23\right)$.
	Since $\lim_{\sigma\downarrow0}\sigma\coth\sigma=1<2\delta$,
	we can choose $0<\sigma=\sigma(\delta)<\frac23$
	such that $\sigma\coth\sigma<2\delta$.
	Let $t_0:=\frac{\sigma}{\delta}$, then
	\begin{align*}
		t_0<\frac{2}{3\delta},\quad
		A(t_0)=\coth\sigma,\quad\mbox{and}\quad t_0A(t_0)=\frac{\sigma}{\delta}\coth\sigma<2.
	\end{align*}
	Therefore we can choose a constant $D=D(\delta)>4$ sufficiently close to \(4\) such that $	D\,t_0A(t_0)<8$.
	Equivalently,
	\begin{align*}
		\frac1{Dt_0}>\frac{A(t_0)}8.
	\end{align*}
	By Theorem \ref{t2.3}, applied with this \(D>4\), we have
	\begin{align*}
		u(x,t_0)=H_\phi(p,x,t_0)
		\le
		C(n,\delta)e^{-\mu}
		\exp\left[-\frac{d^2(p,x)}{Dt_0}\right].
	\end{align*}
	Since $f(x)\le\frac14\left(d(p,x)+\sqrt{2n}\right)^2$,
	then
	\begin{align*}
		q(x,t_0)=u(x,t_0)e^{\frac12A(t_0)f(x)}\le C(n,\delta)e^{-\mu}\exp\left[
		-\frac{d^2(p,x)}{Dt_0}+\frac{A(t_0)}8\left(d(p,x)+\sqrt{2n}\right)^2\right].
	\end{align*}
	Note that $\frac1{Dt_0}>\frac{A(t_0)}8$,
	then the coefficient of \(d^2(p,x)\) in the exponent is strictly negative. Therefore the exponent is bounded above by a constant depending only on \(n\) and \(\delta\). Consequently, for all $x\in M^n$,
	\begin{align*}
		q(x,t_0)\le C(n,\delta)e^{-\boldsymbol{\mu}}.
	\end{align*}
	It remains to verify the growth assumption in Lemma \ref{l2.5}. Let $	t_*:=\frac{2}{3\delta}$ and
	first consider $t\in[t_0,t_*].$
	By Theorem \ref{t2.3},
	\begin{align*}
		q(x,t)
		\leq
		C(n,\delta)e^{-\mu}
		\exp\left[
		-\frac{d^2(p,x)}{Dt}
		+
		\frac12A(t)f(x)\right].
	\end{align*}
	For every \(\eta>0\), the estimate $f(x)\le\frac14\left(d(p,x)+\sqrt{2n}\right)^2$
	implies
	\begin{align*}
		d^2(p,x)\ge 4(1-\eta)f(x)-C(n,\eta).
	\end{align*}
	Thus
	\begin{align*}
		q(x,t)
		\leq
		C(n,\delta,\eta)e^{-\mu}
		\exp\left[
		\left(
		\frac12A(t)-\frac{4(1-\eta)}{Dt}
		\right)f(x)
		\right].
	\end{align*}
	We claim that \(D>4\) and \(\eta>0\) can be chosen so that (Obviously, this $D$ does not violate the choice of $D$ as above )
	\begin{align*}
		\sup_{t\in[t_0,t_*]}
		\left(
		\frac12A(t)-\frac{4(1-\eta)}{Dt}
		\right)<1.
	\end{align*}
Indeed, let $z=\delta t$, then \(0<z\le\frac23\). When \(D=4\) and \(\eta=0\), we get
	\begin{align*}
		\frac12A(t)-\frac{4}{Dt}= \frac12\coth z-\frac{\delta}{z}.
	\end{align*}
	Since \(\delta>\frac12\),
	\begin{align*}
		\frac12\coth z-\frac{\delta}{z}<\frac12\left(\coth z-\frac1z\right).
	\end{align*}
	For \(z>0\), it's easy to check $\coth z-1/z<z$,
	therefore for \(0<z\le\frac23\),
	\begin{align*}
		\frac12\left(\coth z-\frac1z\right)<\frac13<1.
	\end{align*}
	Choosing \(D>4\) sufficiently close to \(4\), and then setting \(\eta>0\) sufficiently small, the claimed strict inequality remains true. Hence there exists a constant $\varepsilon_1=\varepsilon_1(\delta)>0$
	such that for all \(x\in M\) and \(t\in[t_0,t_*]\),
	\begin{align*}
		q(x,t)\le C(n,\delta)e^{-\mu}e^{(1-\varepsilon_1)f(x)}.
	\end{align*}
	
	Next consider $t\ge t_*$.
	Since $t_*=\frac{2}{3\delta}>1$,
	Theorem \ref{t2.1} gives
	\begin{align*}
		H_\phi(p,x,t)\le C(n)e^{-\boldsymbol{\mu}}.
	\end{align*}
	Therefore
	\begin{align*}
		q(x,t)\leq C(n)e^{-\mu}\exp\left[\frac12A(t)f(x)\right].
	\end{align*}
	Since \(A(t)\) is decreasing in \(t\),
	\begin{align*}
		\frac12A(t)\le \frac12A(t_*)=\frac12\coth\frac23<1.
	\end{align*}
	Thus there exists $\varepsilon_2>0$ such that for all \(x\in M\) and \(t\ge t_*\),
	\begin{align*}
		q(x,t)\le C(n)e^{-\mu}e^{(1-\varepsilon_2)f(x)}
	\end{align*}
	Combining the two time ranges, for every \(T>t_0\), there exist constants $C$ and $\varepsilon:=\min\{\varepsilon_1,\varepsilon_2\}>0$
	such that on $M\times[t_0,T]$,
	\begin{align*}
		q(x,t)\le Ce^{(1-\varepsilon)f(x)}.
	\end{align*}
	Hence \(q\) satisfies all assumptions of Lemma \ref{l2.5} on \(M\times[t_0,T]\). Recall
	\begin{align*}
		\partial_t q-\Delta q+A(t)\langle\nabla q,\nabla f\rangle\le0\quad \mbox{and}\quad q(x,t_0)\le C(n,\delta)e^{-\boldsymbol{\mu}},
	\end{align*}
	Now Lemma \ref{l2.5} gives that for all \(x\in M\) and \(t\in[t_0,T]\),
	\begin{align*}
		q(x,t)\le C(n,\delta)e^{-\mu}
	\end{align*}
	Since \(T>t_0\) is arbitrary, hence for all \(x\in M\) and \(t\ge t_0\), we conclude
	\begin{align*}
		q(x,t)\le C(n,\delta)e^{-\mu}.
	\end{align*}
	Note that $t_0<\frac{2}{3\delta}$, then the desired estimate follows for all $t\ge\frac{2}{3\delta}$.
	This completes the proof .
\end{proof}
Since $f(p)\leq\frac{n}{2}$ and $-(\tanh\delta t)^{-1}+1=-\frac{2e^{-\delta t}}{e^{\delta t}-e^{-\delta t}}=-\frac{2e^{-2\delta t}}{1-e^{-2\delta t}}$, hence by $	H_{f}=H_{\phi} e^{\frac{f(x)+f(y)}{2}}$, we immediately have
\begin{theorem}[\textbf{=Theorem \ref{t1.2'}}]\label{t3.12}
	For any Ricci shrinker $(M^n,g,f)$, $x\in M^n$, $\delta\in\left(\frac{1}{2}, \frac{2}{3}\right)$ and $t\geq \frac{2}{3\delta}$, there exists a constant $C=C(n,\delta)$ such that 
	\begin{equation*}
		H_f(p,x,t)\leq Ce^{\boldsymbol{-\mu}}e^{-\frac{e^{-2\delta t}}{1-e^{-2\delta t}}f(x)}.
	\end{equation*}
	For $0<t\leq\frac{4}{3}$ and $4<D<6$, by Theorem \ref{t2.3}, there exists a constant $C=C(n,D)$ such that 
	\begin{equation*}
		H_f(p,x,t)\leq C\gamma(t)\exp\left(-\frac{d^2(p,x)}{Dt}+\frac{f(x)}{2}\right). 
	\end{equation*}
\end{theorem}

\begin{remark}\label{r3.13}
	For Gaussian shrinker $\left(\mathbb{R}^n, g_{E}, e^{-\frac{x^2}{4}}dx\right)$, $\delta$ can be taken as $\frac{1}{2}$, so the above estimates are almost sharp for short and long time.
\end{remark}

\section{Ricci flow and diffeomorphisms}\label{sec-diff}
In this section, we are going to explain the relation between $H_f$ and spacetime heat kernel $H(x,t;y,s)$ under the self-similar Ricci flow induced by a Ricci shrinker.

Let \((M^n,\bar g,f)\) be a normalized Ricci shrinker satisfying
\begin{align*}
	\overline{Rc}+\overline{\nabla}^2f=\frac12\bar g.
\end{align*}
Let \(\psi_s:M^n\to M^n\) be the family of diffeomorphisms generated by $X=\bar\nabla f$.
Define
\begin{align*}
	\hat g_s:=\psi_s^*\bar g,\quad -\infty<s<\infty,
\end{align*}
and
\begin{align*}
	g_t:=(1-t)\hat g_{\ln\frac1{1-t}}
	=e^{-s}\hat g_s,\quad s=\ln\frac1{1-t}.
\end{align*}
Since
\begin{align*}
	\mathcal L_{\bar\nabla f}\bar g=2\overline{\nabla}^2f
	=\bar g-2\overline{Rc},
\end{align*}
we have
\begin{align*}
	\partial_s\hat g_s
	= \psi_s^*(\mathcal L_{\bar\nabla f}\bar g)
	= \hat g_s-2Rc_{\hat g_s}.
\end{align*}
Moreover,
\begin{align*}
	\frac{ds}{dt}=\frac1{1-t}=e^s.
\end{align*}
Thus
\begin{align*}
	\partial_tg_t
	= -\hat g_s+(1-t)\frac{ds}{dt}\partial_s\hat g_s 
	= -\hat g_s+\hat g_s-2Rc_{\hat g_s} 
	= -2Rc_{g_t}.
\end{align*}
Hence \(\hat g_s\) solves the modified Ricci flow, and \(g_t\) solves the Ricci flow.

\subsection*{Heat and conjugate heat equations under rescaling}

Fix \(x,y\in M^n\) and two time points $s_0<t_1<1$.
Set
\begin{align*}
u(z,t)&:=H(z,t;y,s_0),\quad s_0<t<t_1,\\
w(z,s)&:=H(x,t_1;z,s),\quad s_0<s<t_1.
\end{align*}
Thus
\begin{align*}
	\left(\partial_t-\Delta_{g_t}\right)u = 0,\quad
	\left(-\partial_s-\Delta_{g_s}+R_{g_s}\right)w = 0.
\end{align*}
For every intermediate time \(r\in(s_0,t_1)\), the semigroup property
gives
\begin{align*}
	\int_{M^n} u(z,r)w(z,r)\,dv_{g_r}(z)= H(x,t_1;y,s_0).
\end{align*}
Moreover, stochastic completeness of the conjugate heat kernel gives
\begin{align*}
	\int_{M^n} w(z,s)\,dv_{g_s}(z)=1,\quad s_0<s<t_1.
\end{align*}
Let $\sigma_0:=\ln\frac{1}{1-s_0}$ and $\sigma_1:=\ln\frac{1}{1-t_1}$.
For an intermediate time \(r\in(s_0,t_1)\), write
\begin{align*}
	\sigma:=\ln\frac{1}{1-r},\qquad
	r=1-e^{-\sigma},\quad \sigma_0<\sigma<\sigma_1,
\end{align*}
and define
\begin{align*}
	\widetilde u(z,\sigma) := u\left(z,1-e^{-\sigma}\right),\quad
	\widetilde w(z,\sigma) := w\left(z,1-e^{-\sigma}\right).
\end{align*}
Recall that
\begin{align*}
	g_r=e^{-\sigma}\widehat g_\sigma,\quad
	\Delta_{g_r}=e^\sigma\Delta_{\widehat g_\sigma},\quad
	R_{g_r}=e^\sigma R_{\widehat g_\sigma},\quad
	\frac{dr}{d\sigma}=e^{-\sigma}.
\end{align*}
Therefore,
\begin{align*}
	\partial_\sigma\widetilde u= e^{-\sigma}\partial_tu= e^{-\sigma}\Delta_{g_r}u= \Delta_{\widehat g_\sigma}\widetilde u,
\end{align*}
and hence
\begin{align*}
	\left(\partial_\sigma-\Delta_{\widehat g_\sigma}\right)\widetilde u=0.
\end{align*}

For the conjugate heat kernel, we similarly have
\begin{align*}
	\partial_\sigma\widetilde w= e^{-\sigma}\partial_sw= e^{-\sigma}\left(-\Delta_{g_r}w+R_{g_r}w\right)= -\Delta_{\widehat g_\sigma}\widetilde w + R_{\widehat g_\sigma}\widetilde w.
\end{align*}
It is convenient to introduce the backward logarithmic time $\theta:=\sigma_1-\sigma$.
Thus \(\theta\downarrow0\) corresponds to \(r\uparrow t_1\), and $	\partial_\theta=-\partial_\sigma$.
Regarding \(\widetilde w\) as a function of \((z,\theta)\), we obtain
\begin{align*}
	\left(\partial_\theta-\Delta_{\widehat g_\sigma}+R_{\widehat g_\sigma}\right)\widetilde w=0,\quad
	\sigma=\sigma_1-\theta.
\end{align*}
To preserve the heat/conjugate-heat pairing. We therefore define for $\sigma=\sigma_1-\theta$ that
\begin{align*}
	\widehat u(z,\sigma):=\widetilde u(z,\sigma),\quad
	\widehat w(z,\theta) := e^{-\frac n2\sigma}\widetilde w(z,\sigma),\quad .
\end{align*}
Then
\begin{align*}
	\left(\partial_\sigma-\Delta_{\widehat g_\sigma}\right)\widehat u=0.
\end{align*}
Furthermore,
\begin{align*}
	\partial_\theta\widehat w
	= \frac n2e^{-\frac n2\sigma}\widetilde w + e^{-\frac n2\sigma}\partial_\theta\widetilde w
	= \Delta_{\widehat g_\sigma}\widehat w - R_{\widehat g_\sigma}\widehat w + \frac n2\widehat w,
\end{align*}
and hence
\begin{align*}
	\left(\partial_\theta-\Delta_{\widehat g_\sigma}+R_{\widehat g_\sigma}-\frac n2\right)\widehat w=0.
\end{align*}
The rescaling factors now cancel exactly. Indeed, at the intermediate
time \(r=1-e^{-\sigma}\),
\begin{align*}
	\int_{M^n} \widehat u(z,\sigma)\widehat w(z,\theta)\, dv_{\widehat g_\sigma}(z)
	&= \int_{M^n} \widetilde u(z,\sigma)\left(e^{-\frac n2\sigma}\widetilde w(z,\sigma)\right) e^{\frac n2\sigma}\,dv_{g_r}(z)\\
	&= \int_{M^n} u(z,r)w(z,r)\,dv_{g_r}(z)\\
	&= H(x,t_1;y,s_0).
\end{align*}
Similarly,
\begin{align*}
	\int_{M^n} \widehat w(z,\theta)\, dv_{\widehat g_\sigma}(z)
	&= \int_{M^n} e^{-\frac n2\sigma}\widetilde w(z,\sigma) e^{\frac n2\sigma}\,dv_{g_r}(z)\\
	&= \int_{M^n} w(z,r)\,dv_{g_r}(z)\\
	&= 1.
\end{align*}

\subsection*{Pullback to the fixed shrinker metric \(\bar g\)}

Recall that $\widehat g_\sigma=\psi_\sigma^*\bar g$ and $\bar g=\left(\psi_\sigma^{-1}\right)^*\widehat g_\sigma$.
Define for $\sigma=\sigma_1-\theta$ that
\begin{align*}
	\bar u(z,\sigma) = \widehat u\left(\psi_\sigma^{-1}(z),\sigma\right),\quad
	\bar w(z,\theta) = \widehat w\left(\psi_\sigma^{-1}(z),\theta\right).
\end{align*}
Put $q:=\psi_\sigma^{-1}(z)$,
and denote by \(\widehat X\) the vector field on \((M^n,\widehat g_\sigma)\) corresponding to $	X=\bar\nabla f$
under the diffeomorphism \(\psi_\sigma\). Since \(\psi_\sigma\) is generated by \(X\), we have
\begin{align*}
	\partial_\sigma\bar u(z,\sigma)
	&= \partial_\sigma\widehat u(q,\sigma) - \left\langle \widehat\nabla\widehat u(q,\sigma), \widehat X(q,\sigma) \right\rangle_{\widehat g_\sigma}\\
	&= \Delta_{\widehat g_\sigma}\widehat u(q,\sigma) - \left\langle \widehat\nabla\widehat u(q,\sigma), \widehat X(q,\sigma) \right\rangle_{\widehat g_\sigma}\\
	&= \bar\Delta\bar u(z,\sigma) - \left\langle \bar\nabla f(z), \bar\nabla\bar u(z,\sigma) \right\rangle_{\bar g}.
\end{align*}
Therefore
\begin{align*}
	\left(\partial_\sigma-\bar\Delta_f\right)\bar u=0.
\end{align*}

For \(\bar w\), it is useful first to regard it as a function of \(\sigma\). The equation for \(\widehat w\) is equivalently
\begin{align*}
	\partial_\sigma\widehat w = -\Delta_{\widehat g_\sigma}\widehat w + R_{\widehat g_\sigma}\widehat w - \frac n2\widehat w.
\end{align*}
Applying the same pullback calculation gives
\begin{align*}
	\partial_\sigma\bar w
	&= -\bar\Delta\bar w - \left\langle \bar\nabla f,\bar\nabla\bar w \right\rangle_{\bar g} + \left(\bar R-\frac n2\right)\bar w.
\end{align*}
Since \(\theta=\sigma_1-\sigma\), thus $\partial_\theta=-\partial_\sigma$.
and hence
\begin{align*}
	\partial_\theta\bar w
	= \bar\Delta\bar w + \left\langle \bar\nabla f,\bar\nabla\bar w \right\rangle_{\bar g} + \left(\frac n2-\bar R\right)\bar w.
\end{align*}
Equivalently, for $\bar\Delta_{-f} := \bar\Delta + \left\langle \bar\nabla f,\bar\nabla\cdot \right\rangle_{\bar g}$, we have
\begin{align*}
	\left(\partial_\theta-\bar\Delta_{-f}+\bar R-\frac n2\right)\bar w=0,
\end{align*}

The heat/conjugate-heat pairing is preserved under this pullback. In fact, recall $\theta=\sigma_1-\sigma$, a direct change of variables gives
\begin{align*}
	\int_{M^n} \bar u(z,\sigma)\bar w(z,\sigma_1-\sigma)\,d\bar v(z)
	= \int_{M^n} \widehat u(q,\sigma)\widehat w(q,\theta)\, dv_{\widehat g_\sigma}(q)= H(x,t_1;y,s_0).
\end{align*}
Similarly,
\begin{align*}
	\int_{M^n} \bar w(z,\theta)\,d\bar v(z)
	= \int_{M^n} \widehat w(q,\theta)\, dv_{\widehat g_\sigma}(q)= 1.
\end{align*}
Set $\bar x:=\psi_{\sigma_1}(x)$,
the terminal Dirac normalization of the original conjugate heat kernel is also preserved by the rescaling and the pullback. More precisely, for every \(\eta\in C_0^\infty(M^n)\),
\begin{align*}
	\int_{M^n} \eta(z)\bar w(z,\theta)\,d\bar v(z)
	&= \int_{M^n} \eta\left(\psi_\sigma(q)\right) w\left(q,1-e^{-\sigma}\right)\, dv_{g_{1-e^{-\sigma}}}(q).
\end{align*}
As \(\theta\downarrow0\), we have $\sigma\uparrow\sigma_1$ and $1-e^{-\sigma}\uparrow t_1$.
Therefore
\begin{align*}
	\lim_{\theta\downarrow0} \int_{M^n} \eta(z)\bar w(z,\theta)\,d\bar v(z)
	= \eta\left(\psi_{\sigma_1}(x)\right)
	= \eta(\bar x).
\end{align*}
In other words,
\begin{align*}
	\bar w(\cdot,\theta)\,d\bar v \longrightarrow \delta_{\bar x}\quad \text{as}\quad\theta\downarrow0.
\end{align*}

\subsection*{The relation with the weighted heat kernel \(H_f\)}

Define $h(z,\theta):=e^{f(z)}\bar w(z,\theta)$.
We claim that \(h\) solves the \(f\)-heat equation $\left(\partial_\theta-\bar\Delta_f\right)h=0$.

First, we have $\partial_\theta h = e^f\partial_\theta\bar w$.
On the other hand, we compute
\begin{align*}
	\bar\Delta_f\left(e^f\bar w\right)
	&= \bar\Delta\left(e^f\bar w\right) - \left\langle \bar\nabla f,\bar\nabla\left(e^f\bar w\right) \right\rangle_{\bar g}\\
	&= e^f\left[\bar\Delta\bar w + 2\left\langle \bar\nabla f,\bar\nabla\bar w \right\rangle_{\bar g}
	+ \bar w\bar\Delta f + \bar w|\bar\nabla f|_{\bar g}^2 \right]- e^f\left[\left\langle \bar\nabla f,\bar\nabla\bar w \right\rangle_{\bar g} + \bar w|\bar\nabla f|_{\bar g}^2\right]\\
	&= e^f\left[\bar\Delta\bar w + \left\langle \bar\nabla f,\bar\nabla\bar w \right\rangle_{\bar g} + \bar w\bar\Delta f\right].
\end{align*}
By
\begin{align*}
	\bar\Delta f+\bar R=\frac n2,\quad	\partial_\theta\bar w
	= \bar\Delta\bar w + \left\langle \bar\nabla f,\bar\nabla\bar w \right\rangle_{\bar g} + \left(\frac n2-\bar R\right)\bar w,
\end{align*}
we obtain
\begin{align*}
	\bar\Delta_f\left(e^f\bar w\right)
	= e^f\left[\bar\Delta\bar w + \left\langle \bar\nabla f,\bar\nabla\bar w \right\rangle_{\bar g} + \left(\frac n2-\bar R\right)\bar w \right]= e^f\partial_\theta\bar w= \partial_\theta h.
\end{align*}
Thus $	\left(\partial_\theta-\bar\Delta_f\right)h=0$.
Moreover, note that  $h(z,\theta)e^{-f(z)}\,d\bar v(z) = \bar w(z,\theta)\,d\bar v(z)$,
so
\begin{align*}
	\int_{M^n} h(z,\theta)e^{-f(z)}\,d\bar v(z) = 1
\end{align*}
and, for every \(\eta\in C_0^\infty(M^n)\),
\begin{align*}
	\lim_{\theta\downarrow0} \int_{M^n} \eta(z)h(z,\theta)e^{-f(z)}\,d\bar v(z)
	= \lim_{\theta\downarrow0} \int_{M^n} \eta(z)\bar w(z,\theta)\,d\bar v(z)= \eta(\bar x).
\end{align*}
Therefore \(h\) is the weighted heat kernel with pole at \(\bar x\) at time \(\theta=0\). 
Since $H_f(\bar x,\cdot,\theta)$ is the minimal nonnegative fundamental solution, we have $H_f(\bar x,\cdot,\theta)\leq h(\cdot,\theta)$. Both sides have unit mass with respect to $e^{-f}dv_{\bar g}$,
and hence they agree identically. By symmetry, we conclude
\begin{align*}
	h(z,\theta) = H_f(z,\bar x,\theta) = H_f(\bar x,z,\theta).
\end{align*}
Since \(h=e^f\bar w\), it follows that for $\bar x=\psi_{\sigma_1}(x)$, $\bar w(z,\theta) = H_f(\bar x,z,\theta)e^{-f(z)}$.

We now evaluate this identity at the initial time \(s_0\). Put $\theta_0:=\sigma_1-\sigma_0 = \ln\frac{1-s_0}{1-t_1}$.
By the definitions of \(\bar w\) and \(\widehat w\),
\begin{align*}
	&\bar w\left(\psi_{\sigma_0}(y),\theta_0\right)\\
	&= \widehat w(y,\theta_0)= e^{-\frac n2\sigma_0}\widetilde w(y,\sigma_0)= e^{-\frac n2\sigma_0}w(y,s_0)= e^{-\frac n2\sigma_0}H(x,t_1;y,s_0).
\end{align*}
On the other hand, the preceding identification with the weighted heat kernel gives
\begin{align*}
	\bar w\left(\psi_{\sigma_0}(y),\theta_0\right)
	= H_f\left( \psi_{\sigma_1}(x), \psi_{\sigma_0}(y), \sigma_1-\sigma_0 \right)
	\exp\left[ -f\left(\psi_{\sigma_0}(y)\right) \right].
\end{align*}
Combining these two identities, we obtain
\begin{align*}
	H(x,t_1;y,s_0)
	&= e^{\frac n2\sigma_0} H_f\left( \psi_{\sigma_1}(x), \psi_{\sigma_0}(y), \sigma_1-\sigma_0 \right)
	\exp\left[ -f\left(\psi_{\sigma_0}(y)\right) \right].
\end{align*}
Since
\begin{align*}
	e^{\sigma_0}=\frac{1}{1-s_0},\quad
	\sigma_1=\ln\frac{1}{1-t_1},\quad
	\sigma_0=\ln\frac{1}{1-s_0},
\end{align*}
we conclude that
\begin{align*}
	H(x,t_1;y,s_0)
	&= (1-s_0)^{-\frac n2}
	H_f\left( \psi_{\ln\frac{1}{1-t_1}}(x), \psi_{\ln\frac{1}{1-s_0}}(y), \ln\frac{1-s_0}{1-t_1} \right)\exp\left[ -f\left( \psi_{\ln\frac{1}{1-s_0}}(y) \right) \right].
\end{align*}

To sum up, we have the following proposition.
\begin{proposition}\label{p6.1}
	Let $(M^n,g(t))_{t<1}$ be the ancient Ricci flow associated with a Ricci shrinker $(M^n,g,f)$ and $H(x,t;y,s)_{s<t<1}$ be the spacetime heat kernel. Denote the family of diffeomorphisms generated by $\nabla f$ as $\psi_\theta, \theta\in(-\infty,+\infty)$, then
	\begin{align*}
		H(x,t;y,s)=(1-s)^{-n/2}H_f\left(\psi_{\ln\frac1{1-t}}(x), \psi_{\ln\frac1{1-s}}(y), \ln\frac{1-s}{1-t}\right)\exp\left[-f\left(\psi_{\ln\frac1{1-s}}(y)\right)\right].
	\end{align*}
	Conversely, for any $\tau\in(-\infty,+\infty)$ and $t>0$, 
	\begin{align*}
		H_f(x,y,t)e^{-f(y)}=e^{-n\tau/2}H\left(\psi_{\tau+t}^{-1}(x),1-e^{-(\tau+t)}; \psi_\tau^{-1}(y),1-e^{-\tau}\right).
	\end{align*}
\end{proposition}
Especially, by \cite[Theorem 15]{LW}, we have
\begin{proposition}\label{p6.2}
	For any Ricci shrinker $(M^n,g,f)$, $x,y\in M^n$, and $t>0$,
	\begin{align*}
		H_f(x,y,t)e^{-f(y)}\leq\frac{e^{-\boldsymbol{\mu}}}
		{\left(4\pi\left(1-e^{-t}\right)\right)^{\frac n2}}.
	\end{align*}
	By the symmetry of $H_f$, we also have
	\begin{align*}
		H_f(y,x,t)e^{-f(y)}\leq\frac{e^{-\boldsymbol{\mu}}}
		{\left(4\pi\left(1-e^{-t}\right)\right)^{\frac n2}}.
	\end{align*}
\end{proposition}

\section{Nash entropy}\label{sec-N}
Throughout this section, unless otherwise specified, all spatial derivatives of $H_f(x,y,t)$, $h_f(x,y,t)$, and $b_f(x,y,t)$ are taken with respect to the $y$-variable.
Let $g$ be the fixed shrinker metric and $H_f(x,y,t)=\frac{1}{\zeta(t)}e^{-h_f(x,y,t)}$. For Ricci shrinkers, we set $\zeta(t)=\left[4\pi\left(1-e^{-t} \right) \right]^{\frac{n}{2}}$. Formally, in light of the probability measure $dv_{H_f}(y)=H_f(x,y,t)e^{-f(y)}dv(y)=\frac{1}{\zeta(t)}e^{-h_f(x,y,t)-f(y)}dv(y)$, define
the weighted Nash entropy $\mathcal{N}_f(x,t)$ at $(x,t)$ by 
\begin{align*}
	\mathcal{N}_f(x,t)&=\int_{M^n}\left[f(y)+h_f(x,y,t)\right]H_f(x,y,t)e^{-f(y)}dv(y)-\frac{n}{2}\\
	&=\int_{M^n}\left[f(y)-\ln H_f(x,y,t)-\ln \zeta(t)\right]dv_{H_f}(y)-\frac{n}{2}.
\end{align*}
In fact, by our previous arguments of heat kernel, this definition is equivalent to that of the Nash entropy under Ricci flow. Let $H(x,t;y,s)$ be the Ricci flow conjugate heat kernel starting from $(x,t)$, $s<t<1$. Set
\begin{align*}
	\sigma_s:=\ln\frac1{1-s},\quad\sigma_t:=\ln\frac1{1-t},	\quad\theta:=\sigma_t-\sigma_s
	= \ln\frac{1-s}{1-t},\quad z:=\psi_{\sigma_s}(y),\quad w:=\psi_{\sigma_t}(x).
\end{align*}
As before, $\psi_\sigma$ is generated by $X=\nabla f$, and the associated flow is
\begin{align*}
	\hat g_\sigma=\psi_\sigma^*\bar g,\quad
	g_t=(1-t)\hat g_{\ln\frac1{1-t}}.
\end{align*}
By proposition \ref{p6.1}, we have
\begin{align*}
	H(x,t;y,s)=(1-s)^{-n/2}e^{-f(z)}H_f(z,w,\theta).
\end{align*}
Since $H_f$ is symmetric, we may write $H_f(z,w,\theta)=H_f(w,z,\theta)$. By $	t-s=(1-s)\left(1-e^{-\theta}\right)$, 
\begin{align*}
	H(x,t;y,s)&=(1-s)^{-n/2}e^{-f(z)}
	\left(4\pi\left(1-e^{-\theta}\right)\right)^{-n/2}
	e^{-h_f(w,z,\theta)}\\
	&=
	\left(4\pi(t-s)\right)^{-n/2}
	\exp\left[-\left(f(z)+h_f(w,z,\theta)\right)\right],
\end{align*}
Thus, if the Ricci flow conjugate heat kernel is written as
\begin{align*}
	H(x,t;y,s)=(4\pi(t-s))^{-n/2}e^{-b_{x,t}(y,s)},
\end{align*}
then
\begin{align*}
	b_{x,t}(y,s)=
	f\left(\psi_{\ln\frac1{1-s}}(y)\right)
	+
	h_f\left(
	\psi_{\ln\frac1{1-t}}(x),
	\psi_{\ln\frac1{1-s}}(y),
	\ln\frac{1-s}{1-t}
	\right).
\end{align*}
Moreover, as we discussed previously, the corresponding heat kernel measures satisfy
\begin{align*}
	H(x,t;y,s)dv_{g_s}(y)=H_f(w,z,\theta)e^{-f(z)}dv(z).
\end{align*}

Recall the definition of pointed Nash entropy and $\mathcal{W}$ entropy under Ricci flow:
\begin{align*}
	\mathcal N_{(x,t)}(t-s):&=
	\int_{M^n} b_{x,t}(y,s)H(x,t;y,s)dv_{g_s}(y)-\frac n2,\\
	\mathcal W_{(x,t)}(t-s):&=
	\int_{M^n} \left[(t-s)\left(\left|\nabla_yb_{x,t}\right|_{g_s}^2+R_{g_s}\right)+b_{x,t}-n\right]H(x,t;y,s)dv_{g_s}(y).
\end{align*}
Then, changing variables $z=\psi_{\sigma_s}(y)$, we conclude
\begin{align*}
	\mathcal N_{(x,t)}(t-s)
	&=\int_{M^n}\left(f(z)+h_f(w,z,\theta)\right)H_f(w,z,\theta)e^{-f(z)}dv(z)-\frac n2=\mathcal N_f(w,\theta).
\end{align*}
Therefore the direct entropy relation is
\begin{align}\label{7.1*}
	\mathcal N_{(x,t)}(t-s)=\mathcal N_f\left(\psi_{\ln\frac1{1-t}}(x),\ln\frac{1-s}{1-t}\right).
\end{align}
Conversely, fix any $a\in\mathbb R$ and $\theta>0$.
By similar arguments, we have
\begin{align*}
	\mathcal	N_f(x,\theta)=\mathcal N_{\left(\psi_{a+\theta}^{-1}(x),\,1-e^{-(a+\theta)}\right)}
	\left(e^{-a}\left(1-e^{-\theta}\right)\right).
\end{align*}

Next we turn to Perelman's $\mathcal{W}$ entropy. For some smooth function $\phi$ and $\tau>0$, let $u^2=\frac{e^{-\phi}}{(4\pi\tau)^{n/2}}$ and $\int_{M^n}u^2dv=1$,
define
$$\mathcal{W}(g,u,\tau)=\int_{M^n}\tau\left(4|\nabla u|^2+Ru^2\right)-u^2\ln u^2dv-\left(n+\frac n2\ln(4\pi\tau)\right).$$ 
For a general Ricci shrinker $(M^n,g,f)$, we define the $\boldsymbol{\mu}$ functional as $$\boldsymbol{\mu}(g,\tau)=\inf\left\{\mathcal{W}(g,u,\tau)\left|u\in W_*^{1,2}(M)\right.\right\},$$
where 
$$W_*^{1,2}(M^n)=\left\{u\left|\int_{M^n}\left|\nabla u\right| ^2dv<\infty,\int_{M^n} u^2dv=1\hspace*{0.3em}\mathrm{and}\int_{M^n} d^2(p,\cdot)u^2dv<\infty\right\}\right..$$
For more details, especially the validity of these entropies on Ricci shrinkers, see \cite{LW,LW2}.

Now set $u^2(y)=H_f(x,y,t)e^{-f(y)}$, since $\left|\nabla f\right|^2+R=f$, then
\begin{align*}
	&\mathcal{W}(g,u,\tau)\\
	&=\int_{M^n}\tau\left(\frac{\left|\nabla H_f \right|^2}{H_f^2}-f\right)dv_{H_f}(y)+	\mathcal N_f(x,t)+\frac{n}{2}\ln\left(4\pi\left(1-e^{-t}\right)\right)+n\tau -\left(\frac{n}{2}+\frac n2\ln(4\pi\tau)\right)\\
	&=\int_{M^n}\tau\left(\frac{\left|\nabla H_f \right|^2}{H_f^2}-f\right)dv_{H_f}(y)+	\mathcal N_f(x,t)+\frac{n}{2}\ln\left(1-e^{-t}\right)+n\tau -\left(\frac{n}{2}+\frac n2\ln(\tau)\right).
\end{align*}
On the other hand, since $\Delta_ff=\frac{n}{2}-f$, it follows from integration by parts that
\begin{align*}
	&\partial_t\mathcal{N}_f(x,t)\nonumber\\
	&=\partial_t\int_{M^n}\left[f(y)-\ln H_f(x,y,t)-\ln \zeta(t)\right]dv_{H_f}(y)\nonumber\\
	&=\int_{M^n}\left[-\frac{\partial_tH_f}{H_f}-\frac{\partial_t\zeta(t)}{\zeta(t)}\right]dv_{H_f}(y)+\int_{M^n}\left[f(y)-\ln H_f-\ln \zeta(t)\right]\partial_tH_fe^{-f(y)}dv(y)\nonumber\\
	&=\int_{M^n}\left[-\frac{\Delta_{f,y}H_f}{H_f}-\frac{\partial_t\zeta(t)}{\zeta(t)}\right]H_fe^{-f(y)}dv(y)\nonumber\\
	&\quad+\int_{M^n}\left[f(y)-\ln H_f-\ln \zeta(t)\right]\Delta_{f,y}H_fe^{-f(y)}dv(y)\nonumber\\
	&=\int_{M^n}\left[-\frac{n}{2}\left(\frac{1}{e^t-1}\right)\right]H_fe^{-f(y)}dv(y)+\int_{M^n}\left[\Delta_{f,y}f(y)+\frac{\left|\nabla H_f \right|^2}{H_f^2}\right]H_fe^{-f(y)}dv(y)\nonumber\\
	&=\int_{M^n}\left[\frac{\left|\nabla H_f \right|^2}{H_f^2}-f\right]dv_{H_f}(y)+\frac{n}{2}-\frac{n}{2}\left(\frac{1}{e^t-1}\right).
\end{align*}
Therefore, if we set $\tau(t)=1-e^{-t}$, we have $\mathcal{W}\left(g,\sqrt{H_f}e^{-\frac{f}{2}}, 1-e^{-t}\right)=\tau\partial_t\mathcal{N}_f+\mathcal{N}_f$. So it's natural to define the weighted $\mathcal{W}$ entropy $\mathcal{W}_f(x,t)$ as
$$\mathcal{W}_f(x,t)=\mathcal{W}\left(g,\sqrt{H_f}e^{-\frac{f}{2}}, 1-e^{-t}\right).$$
Following Perelman \cite{P}, we define the differential Harnack quantity
$$w_f=\tau\left(2\Delta b-\left|\nabla b\right|^2+R\right)+b-n,$$
where $\tau=1-e^{-t}$ and $\frac{e^{-b_f}}{\left(4\pi\tau\right)^{\frac{n}{2}}}=H_fe^{-f}$, i.e. $b_f=f-\ln H_f-\frac{n}{2}\ln4\pi\tau$. Moreover, by shrinker equation (\ref{1.1}), 
\begin{align*}
	2\Delta b_f=2\left(\Delta f-\Delta\ln H_f\right)=2\left(\frac{n}{2}-R-\frac{\Delta H_f}{H_f}+\frac{\left|\nabla H_f \right|^2}{H_f^2}\right),\\
	\left|\nabla b_f\right|^2=\left|\nabla f-\nabla\ln H_f\right|^2=\left|\nabla f\right|^2+\frac{\left|\nabla H_f \right|^2}{H_f^2}-2\frac{\left\langle\nabla f,\nabla H_f\right\rangle}{H_f},
\end{align*}
hence
\begin{align*}
	2\Delta b_f-\left|\nabla b_f\right|^2+R=n-f-\frac{2\Delta_fH_f}{H_f}+\frac{\left|\nabla H_f \right|^2}{H_f^2},
\end{align*}
and
\begin{align*}
	w_f=\tau\left(n-f-\frac{2\Delta_fH_f}{H_f}+\frac{\left|\nabla H_f \right|^2}{H_f^2}\right)+f-\ln H_f-\frac{n}{2}\ln4\pi\tau-n.
\end{align*}
Consequently,  integration by parts immediately gives
$$\mathcal{W}_f(x,t)=\mathcal{W}\left(g,\sqrt{H_f}e^{-\frac{f}{2}}, 1-e^{-t}\right)=\int_{M^n}w_f dv_{H_f}.$$
Same as the relation between $	\mathcal N_f$ and $	\mathcal N_{x,t}$, we also have
\begin{align}\label{7.2*}
	\mathcal W_{(x,t)}(t-s)=W_f\left(\psi_{\ln\frac1{1-t}}(x),\ln\frac{1-s}{1-t}\right).
\end{align}
Conversely, fix any $a\in\mathbb R$ and $\theta>0$.
We have
\begin{align*}
	W_f(x,\theta)=\mathcal W_{\left(\psi_{a+\theta}^{-1}(x),\,1-e^{-(a+\theta)}\right)}
	\left(e^{-a}\left(1-e^{-\theta}\right)\right).
\end{align*}

Fix a spacetime point $(x,t)$ in the Ricci flow and let $\rho:=t-s>0$
be the Ricci flow backward time. Write the conjugate heat kernel as $H(x,t;y,s)=(4\pi\rho)^{-n/2}e^{-b_{x,t}(y,s)}$.
Then the Ricci flow pointed $\mathcal W$-entropy is
\begin{align*}
	\mathcal W_{x,t}(\rho)=\int_{M^n}\left[\rho\left(|\nabla b_{x,t}|_{g_s}^2+R_{g_s}\right)+b_{x,t}-n
	\right]H(x,t;y,s)\,dv_{g_s}(y).
\end{align*}
With this backward-time convention, Perelman's entropy formula gives (cf. \cite[Proposition 3.19]{LW2})
\begin{align*}
	\frac{d}{d\rho}\mathcal W_{x,t}(\rho)=-2\rho\int_{M^n}
	\left|\operatorname{Rc}_{g_s}+\nabla_{g_s}^2 b_{x,t}-\frac{1}{2\rho}g_s\right|_{g_s}^2H(x,t;y,s)\,dv_{g_s}(y).
\end{align*}
Then by standard computations as before, we conclude the following diffeomorphic counterpart of the above formula.
\begin{proposition}\label{t5.3}
	For any Ricci shrinker $(M^n,g, f)$ and $\tau(t)=1-e^{-t}$,
	\begin{align*}
		\frac{d}{dt}\mathcal{W}_f(x,t)=-2\tau\int_{M^n}\left|Hess\ln H_f(y)+\frac{\tau'}{2\tau}g\right|^2dv_{H_f}(y).
	\end{align*}
\end{proposition}

Recall that by the definition,$\mathcal{W}_f\left(x,t\right)=\tau\partial_t\mathcal{N}_f+\mathcal{N}_f$ for $\tau=1-e^{-t}$, therefore as usual,  the monotonicity of $\mathcal{W}_f$ implies the same monotonicity of $\mathcal N_f$.
\begin{corollary}\label{c5.4}
	\begin{align*}
		\boldsymbol{\mu}\leq\mathcal{W}_f(x,t)\leq\mathcal{N}_f(x,t)&=\frac{1}{e^t-1}\int_{0}^{t}e^s\mathcal{W}_f(x,s)ds\leq0.
	\end{align*}
	As a consequence,
	\begin{align*}
		\frac{d}{dt}\mathcal{N}_f(x,t)&\leq 0.
	\end{align*}
\end{corollary}
\begin{proof}
	$\boldsymbol{\mu}\leq\mathcal{W}_{x,t}(\rho)\leq0$ is a standard result, cf. \cite[Corollary 3.22]{LW2}. Thus by (\ref{7.2*}), we have $\boldsymbol{\mu}\leq\mathcal{W}_f(x,t)$.
	Note that
	\begin{align*}
		\frac{d}{dt}\left(\tau\mathcal{N}_f\right)&=\tau'\mathcal{N}_f+\tau\partial_t\mathcal{N}_f\\
		&=e^{-t}\mathcal{N}_f+\tau\partial_t\mathcal{N}_f\\
		&=\tau\partial_t\mathcal{N}_f+\mathcal{N}_f+(e^{-t}-1)\mathcal{N}_f\\
		&=	\mathcal{W}_f-\tau	\mathcal{N}_f.
	\end{align*}
	Solving the above equation with initial condition $\lim\limits_{t\rightarrow0}\left(1-e^{-t}\right)\mathcal N_f(x,t)=0$ gives
	\begin{align*}
		\tau\mathcal{N}_f=e^{-t}\int_{0}^{t}e^s\mathcal{W}_f(x,s)ds,
	\end{align*}
	and this yields
	\begin{align}
		\mathcal{N}_f(x,t)&=\frac{e^{-t}}{1-e^{-t}}\int_{0}^{t}e^s\mathcal{W}_f(x,s)ds\nonumber\\
		&=\frac{1}{e^t-1}\int_{0}^{t}e^s\mathcal{W}_f(x,s)ds.\label{5.6}
	\end{align}
	Here the validity of initial condition comes from (\ref{7.1*}) and the fact that spacetime $\mathcal N_{x,t}(0)=0$. Set
	$$y(t)=\int_{0}^{t}e^s\mathcal{W}_f(x,s)ds-(e^t-1)\mathcal{W}_f(x,t),$$
	then by Proposition \ref{t5.3},
	\begin{align*}
		y(0)=0,\quad y'(t)=e^t\mathcal{W}_f-e^t\mathcal{W}_f-(e^t-1)\partial_t\mathcal{W}_f\geq0.
	\end{align*}
	Consequently, 
	\begin{align}
		\int_{0}^{t}e^s\mathcal{W}_f(x,s)ds\geq(e^t-1)\mathcal{W}_f(x,t).	\label{5.7}
	\end{align}
	Combining (\ref{5.6}) and (\ref{5.7}) completes the proof.
\end{proof}

For Ricci shrinkers, by \cite{LW},  $\boldsymbol{\mu}=\boldsymbol{\mu}(g,1)=\ln\left[ (4\pi)^{-\frac{n}{2}}\int_{M^n}e^{-f}dv\right]$. On the basis of the relation between $\mathcal{N}_f$ and $\mathcal{W}_f$ in the above corollary, we are able to give the long time behavior of entropy.
\begin{corollary}\label{c5.5}
	\begin{align*}
		\lim\limits_{t\longrightarrow+\infty}\mathcal{N}_f(x,t)=\lim\limits_{t\longrightarrow+\infty}\mathcal{W}_f(x,t)=\boldsymbol{\mu}.
	\end{align*}
\end{corollary}
\begin{proof}
	Let $\varphi_r$ be a cut-off function as in \cite[Lemma 3]{LW} such that it is supported in $\left\lbrace f\leq 2r\right\rbrace $, $\left|\nabla_y\varphi_r\right|\leq Cr^{-\frac{1}{2}}$ and $\left|\Delta_y\varphi_r\right|\leq C$ for some uniform constant $C>0$.
	By Proposition \ref{p6.2}, for given $(x,t)$, $H_f$ is bounded, and hence applying the gradient estimate (\ref{7.14*}) to $H_f(x,y,t/2+s)$ with $s\in\left[0,t/2\right] $, $\left| \nabla_y H_f\right|$ is bounded. Note that
	\begin{align*}
		0&=\int_{M^n}\Delta_{f,y}\left(\varphi_rH_f\right)e^{-f(y)}dv(y)\\
		&=\int_{M^n}\left( \varphi_r\Delta_{f,y}H_f+H_f\Delta_{f,y}\varphi_r+2\left\langle\nabla_y\varphi_r,\nabla_yH_f\right\rangle\right) e^{-f(y)}dv(y),
	\end{align*}
	so setting $r\longrightarrow+\infty$, it's easy to see 
	\begin{align*}
		\int_{M^n}\Delta_{f,y}H_fe^{-f(y)}=0.
	\end{align*}
	Then
	\begin{align}
		\mathcal{W}_f(x,t)&=\mathcal{W}\left(g,\sqrt{H_f}e^{-\frac{f}{2}}, 1-e^{-t}\right)\nonumber\\
		&=\int_{M^n}\left[\tau\left(n-f-\frac{2\Delta_fH_f}{H_f}+\frac{\left|\nabla H_f \right|^2}{H_f^2}\right)+f-\ln H_f-\frac{n}{2}\ln4\pi\tau-n\right]H_fe^{-f}dv\nonumber\\
		&=\int_{M^n}\left[\tau\left(n-f+\frac{\left|\nabla H_f \right|^2}{H_f^2}\right)+f-\ln H_f-\frac{n}{2}\ln4\pi\tau-n\right]H_fe^{-f}dv.\label{7.13}
	\end{align}
	
	For the gradient term, we estimate it as follows.	The quantity we need to control is
	\begin{align*}
		I(t):=\int_{M^n} \frac{\left|\nabla_y H_f(x,y,t)\right|^2}{H_f(x,y,t)}e^{-f(y)}\,dv(y)=\int_{M^n}\left|\nabla_y\ln H_f\right|^2dv_{H_f}(y).
	\end{align*}
	After pulling the heat kernel gradient estimate under Ricci flow (cf. \cite[Theorem 4.6]{LW2}) back to the fixed shrinker metric as before and by symmetry of $H_f$, we have
	\begin{align}
		\frac{\left|\nabla_y H_f(x,y,t)\right|}{H_f(x,y,t)}\leq\sqrt{\frac{C}{e^t-1}}
		\sqrt{\ln\left(\frac{C e^{-	\mathcal N_f(y,t)}}{(1-e^{-t})^{n/2}H_f(x,y,t)e^{-f(x)}}\right)},\label{7.14*}
	\end{align}
	where $C=C(n)>0$. Set $\tau(t):=1-e^{-t}$ and define
	\begin{align*}
		\Lambda_t(y):=\ln\left(\frac{C e^{-	\mathcal N_f(y,t)}}{\tau(t)^{n/2}H_f(x,y,t)e^{-f(x)}}	\right).
	\end{align*}
	Then
	\begin{align*}
		\left|\nabla_y\ln H_f(y)\right|^2\leq\frac{C}{e^t-1}\Lambda_t(y),\quad 	I(t)\leq\frac{C}{e^t-1}	\int_{M^n} \Lambda_t(y)dv_{H_f}(y).
	\end{align*}
	We compute
	\begin{align*}
		\int_{M^n}\Lambda_tdv_{H_f}(y)
		&=\ln C+f(x)-\frac n2\ln\tau(t)-\int_{M^n} 	\mathcal N_f(y,t)dv_{H_f}(y)-\int_{M^n} \ln H_f(y)dv_{H_f}(y).
	\end{align*}
	By Corollary \ref{c5.4}, one has $\mathcal{N}_f\geq\mathcal W_f\geq \mu$ and hence $\mathcal N_f(y,t)\geq \mu$, therefore
	\begin{align*}
		-\int_{M^n} \mathcal{N}_f(y,t)dv_{H_f}(y)\leq -\boldsymbol{\mu}.
	\end{align*}
	Next, by the definition of $\mathcal N_f(x,t)$,
	\begin{align*}
		\mathcal{N}_f(x,t)=	\int_{M^n} fdv_{H_f}-\int_{M^n} \ln H_fdv_{H_f}-\frac n2\ln(4\pi\tau(t))-\frac n2.
	\end{align*}
	Therefore
	\begin{align*}
		-\int_{M^n}\ln H_fdv_{H_f}=	\mathcal N_f(x,t)-\int_{M^n} fdv_{H_f}+\frac n2\ln(4\pi\tau(t))+\frac n2.
	\end{align*}
	Since $f\geq0$,
	\begin{align*}
		-\int_{M^n}\ln H_fdv_{H_f}\leq\mathcal{N}_f(x,t)+\frac n2\ln(4\pi\tau(t))+\frac n2.
	\end{align*}
	Substituting this into the previous estimate gives
	\begin{align*}
		\int_M\Lambda_tdv_{H_f}&\leq ln C+f(x)-\boldsymbol{\mu}+\mathcal{N}_f(x,t)+\frac n2\ln(4\pi)+\frac n2.
	\end{align*}
	Now fix any $t_0>0$. Since $\mathcal{N}_f(x,t)$ is nonincreasing in $t$, for all $t\geq t_0$, we have $\mathcal{N}_f(x,t)\leq \mathcal{N}_f(x,t_0)$.
	Therefore, for $t\geq t_0$,
	\begin{align*}
		\int_{M^n}\Lambda_tdv_{H_f}\leq	A_{x,t_0},
	\end{align*}
	where $A_{x,t_0}:=\ln C+f(x)-\mu+\mathcal{N}_f(x,t_0)+\frac n2\ln(4\pi)+\frac n2$.
	Thus
	\begin{align}
		I(t)\leq\frac{C A_{x,t_0}}{e^t-1},\quad t\geq t_0.\label{7.14}
	\end{align}
	In particular, $I(t)\to0$
	as $t\to\infty$.
	
	Moreover, for all sufficiently large $t$, there exists a uniform constant $C=C(x,n,\boldsymbol{\mu})>0$ such that 
	\begin{align}
		Q:=\left|\tau\left(n-f(y)\right)+f(y)-\ln H_f-\frac{n}{2}\ln4\pi\tau-n\right|H_fe^{-f(y)}\leq Ce^{-\frac{f(y)}{2}}\in L^1(M^n,g,dv).\label{7.15}
	\end{align}
Indeed, set \(q_{t}(y):=H_{f}(x,y,t)e^{-f(y)}\), then by Proposition \ref{p6.2} and the symmetry of \(H_{f}\), for \(t\geq 1\), we have $q_{t}(y)\leq C(x,n,\boldsymbol{\mu})e^{-f(y)}$.
Since \(f(y)-\ln H_{f}(x,y,t)=-\ln q_{t}(y)\) and \(1-e^{-1}\leq\tau<1\), we have
\begin{align*}
	Q\leq C\left( (1+f(y))q_{t}(y)+q_{t}(y)\left|\ln q_{t}(y)\right| \right)
	\leq C e^{-f(y)/2},
\end{align*}
where we used \(z\lvert\ln z\rvert\leq C(Z)\sqrt{z}\) for \(0<z\leq Z\).
	
	Meanwhile, by Proposition \ref{p1.2}, $	H_f(x,y,t)\longrightarrow (4\pi)^{-\frac{n}{2}}e^{-\boldsymbol{\mu}} \hspace*{0.5em}as\hspace*{0.5em}t\longrightarrow\infty.$
	Now by \ref{7.14} and (\ref{7.15}), we can set $t\longrightarrow\infty$ in (\ref{7.13}) to obtain
	\begin{align*}
		\lim\limits_{t\longrightarrow+\infty}\mathcal{W}_f(x,t)=\boldsymbol{\mu}e^{\boldsymbol{-\mu}}(4\pi)^{-\frac{n}{2}}\int_{M^n}e^{-f}dv=\boldsymbol{\mu}.
	\end{align*}
	Since $\mathcal{N}_f(x,t)=\frac{1}{e^t-1}\int_{0}^{t}e^s\mathcal{W}_f(x,s)ds$, therefore if $\boldsymbol{\mu}=0$, then the conclusion is obvious. If $\boldsymbol{\mu}<0$, then
	\begin{align*}
		\lim\limits_{t\longrightarrow+\infty}\mathcal{N}_f(x,t)=\lim\limits_{t\longrightarrow+\infty}\frac{\left(\int_{0}^{t}e^s\mathcal{W}_f(x,s)ds\right)'}{\left(e^t-1\right)'}=\lim\limits_{t\longrightarrow+\infty}\frac{e^t\mathcal{W}_f(x,t)}{e^t}=\boldsymbol{\mu}.
	\end{align*}
\end{proof}

\section{Reduced distance}\label{sec-rd}
Under Ricci flow, Perelman's reduced distance provides a sharp lower bound of the conjugate spacetime heat kernel. Thus it's natural to extend this distance to Ricci shrinker by diffeomorphisms and use it to control the lower bound of $H_f$.

Let $\gamma$ be a spacetime curve parametrized by $\tau=-t\in\left[0,\infty\right)$, connecting $(x,0)$ to $(y,-\tau_0)$. Note that
$\gamma(\tau) \in M ^n\times \{-\tau\}$ with metric $g(t)=g(-\tau)$. Then define
\begin{align*}
	\mathcal{L}(\gamma)=\int_0^{\tau_0} \sqrt{\tau} \left(R_{g_t}+|\dot{\gamma}|_{g_t}^2\right) d\tau,\quad L(y,-\tau_0)=\inf_{\gamma}\mathcal{L}(\gamma)\quad\mbox{and}\quad \ell(y,-\tau_0)=\ell_{(x,0)}(y,-\tau_0)=\frac{L(y,-\tau_0)}{2\sqrt{\tau_0}}.
\end{align*}
On the modified Ricci flow solution, we construct curve $\hat{\gamma}$ such that
\begin{align*}
	\hat{\gamma}(\theta)=\gamma(\tau), \quad \theta= \ln (1-t) =\ln (1+\tau), \quad \tau=e^{\theta}-1. 
\end{align*}
Thus, we have
\begin{align*}
	\hat{\gamma}'=\frac{\partial}{\partial \theta} \hat{\gamma}
	=\frac{\partial}{\partial \tau} \gamma \cdot \frac{d\tau}{d\theta}
	=\frac{\partial}{\partial \tau} \gamma \cdot e^{\theta}
	=(1+\tau) \dot{\gamma}
\end{align*}
We define reduced length functional on the spacetime $(M^n, \hat{g}(s))$, which is naturally the pushforward of $\mathcal{L}$ on the Ricci flow spacetime $(M^n, g(t))$.
Recall that $g_t=e^{-s}\hat{g}_s=e^{\theta} \hat{g}_s$. 
Then we have
\begin{align*}
	\hat{\mathcal{L}}(\hat{\gamma})
	&=\mathcal{L}(\gamma)
	=\int_0^{\theta_0} \sqrt{e^{\theta}-1} \cdot e^{-\theta} \left(R_{\hat{g}_s} + |\hat{\gamma}'|_{\hat{g}_s}^2\right) \cdot e^{\theta} d\theta\\
	&=\int_0^{\theta_0} \sqrt{e^{\theta}-1} \left(R_{\hat{g}_s} + |\hat{\gamma}'|_{\hat{g}_s}^2\right) d\theta. 
\end{align*}
As before, we further push this to the static spacetime $(M^n, \bar{g}_s)$ for $\bar{g}_s \equiv \bar{g}$.  In other words, we define
\begin{align*}
	\bar{\gamma}(\theta)=\psi_s(\hat{\gamma})=\psi_{-\theta}(\hat{\gamma}).
\end{align*}
It follows that
\begin{align*}
\bar{\gamma}'(\theta)=\frac{d}{d\theta}\left(\psi_{-\theta}\left(\hat{\gamma}(\theta)\right)\right)=d\psi_{-\theta}\left(\hat{\gamma}'(\theta)\right)-X\left(\bar{\gamma}(\theta)\right).
\end{align*}
Therefore, we have
\begin{align*}
	\bar{\mathcal{L}}(\bar{\gamma})=\hat{\mathcal{L}}(\hat{\gamma})
	&=\int_0^{\theta_0} \sqrt{e^{\theta}-1} \left(R_{\hat{g}_s} + |\hat{\gamma}'|_{\hat{g}_s}^2\right) d\theta\\
	&=\int_0^{\theta_0} \sqrt{e^{\theta}-1} \left(R_{\bar{g}} + |\bar{\gamma}'+X|_{\bar{g}}^2\right) d\theta\\
	&=\int_0^{\theta_0} \sqrt{e^{\theta}-1} \left(R_{\bar{g}} + |X|_{\bar{g}}^2 + |\bar{\gamma}'|_{\bar{g}}^2 + 2 \langle X, \bar{\gamma}' \rangle_{\bar{g}}\right) d\theta.
\end{align*}
Here $\bar{g}$ is the standard Ricci shrinker metric $g$. Note that $X=\nabla f$ and $R+|\nabla f|^2=f$. Therefore, we have 
\begin{align*}
	\bar{\mathcal{L}}(\bar{\gamma})= \int_0^{\theta_0} \sqrt{e^{\theta}-1} \left(f + |\bar{\gamma}'|^2 + 2 \langle \nabla f, \bar{\gamma}' \rangle\right) d\theta. 
\end{align*}
On the other hand, 
\begin{align*}
	\int_0^{\theta_0} 2\sqrt{e^{\theta}-1}  \langle \nabla f, \bar{\gamma}' \rangle d\theta
	&=\left.2\sqrt{e^{\theta}-1}f\left(\bar{\gamma}(\theta)\right)\right|_{\theta=0}^{\theta_0}- \int_0^{\theta_0} f \cdot \frac{e^{\theta}}{\sqrt{e^{\theta}-1}} d\theta\\
	&=2\sqrt{e^{\theta_0}-1} f\left(\bar{\gamma}(\theta_0)\right)- \int_0^{\theta_0} \sqrt{e^{\theta}-1}  \cdot f \cdot \frac{e^{\theta}}{e^{\theta}-1} d\theta. 
\end{align*}
It follows that
\begin{align*}
	\bar{\mathcal{L}}(\bar{\gamma})
	&=2\sqrt{e^{\theta_0}-1}f\left(\bar{\gamma}(\theta_0)\right)
	+\int_0^{\theta_0} 
	\left\{ \sqrt{e^{\theta}-1}\left( |\bar{\gamma}'|^2  - \frac{f}{e^{\theta}-1}\right) \right\} d\theta\\
	&=2\sqrt{\tau_0}f\left(\psi_{-\ln(1+\tau_0)}(y)\right)
	+\int_0^{\theta_0} 
	\left\{ \sqrt{e^{\theta}-1}\left( |\bar{\gamma}'|^2  - \frac{f}{e^{\theta}-1}\right) \right\} d\theta.
\end{align*}

In view of the above discussions, for two spacetime points $(x,0)$, $(y,t)$ on $(M^n, g, f)\times[0,t]$ and the path $\gamma(s)$ connecting them, we define the weighted $\mathcal{L}$ distance as
\begin{align*}
	0\leq\mathcal{L}^x_f(\gamma)
	=2\sqrt{e^{t}-1}f(y) 
	+\int_0^{t} 
	\left\{ \sqrt{e^{s}-1}\left( |\gamma'(s)|^2-\frac{f\left(\gamma(s)\right)}{e^{s}-1}\right) \right\} ds
\end{align*}
and weighted reduced distance as
$$\ell^x_f(y,t)=\inf_{\gamma}\frac{\mathcal{L}^x_f(\gamma)}{2\sqrt{e^t-1}}=f(y)+\inf_{\gamma}\frac{1}{2\sqrt{e^t-1}}\int_0^{t} 
\left\{ \sqrt{e^{s}-1}\left( |\gamma'(s)|^2-\frac{f\left(\gamma(s)\right)}{e^{s}-1}\right) \right\} ds.$$
Here after the pullback, we relabel the endpoint \(\psi_{-\theta_0}(y)\) as \(y\). This is just the diffeomorphic counterpart of reduced distance under Ricci flow.
Recall from standard cut-off arguments of \cite[Theorem 16]{LW}, for the Ricci flow induced by a Ricci shrinker, we have
\begin{align*}
	H(x,0;y,s)\geq \frac{\exp\left(-\ell(y,s)\right)}{\left(4\pi (-s) \right)^{\frac{n}{2}}}.
\end{align*}
By standard diffeomorphisms as before, this implies
\begin{align*}
	H_f(x,y,t)e^{-f(y)}\geq \frac{\exp\left(-\bar{\ell}^x_f(y,t)-f(y)\right)}{\left(4\pi (1-e^{-t}) \right)^{\frac{n}{2}}},
\end{align*}
where $	\bar{\ell}^x_f(y,t):=\ell^x_f(y,t)-f(y)$.
Especially,
\begin{align}\label{8.1}
	H_f(x,y,t)\geq \frac{\exp\left(-\bar{\ell}^x_f(y,t)\right)}{\left(4\pi (1-e^{-t}) \right)^{\frac{n}{2}}},	
\end{align}

\begin{theorem}[\textbf{=Theorem \ref{t1.3'}}]\label{t15.1}
	For any Ricci shrinker $(M^n,g,f)$, $x,y\in M^n$ and $t>0$,
	\begin{align*}
		H_f(x, y, t) \geq \frac{e^{-\bar{\ell}^x_f(y,t)}}{\left(4\pi (1-e^{-t}) \right)^{\frac n2}}.
	\end{align*}
	Moreover, if the inequality becomes an equality somewhere, then the Ricci shrinker is isometric to the standard flat Gaussian shrinker $\left(\mathbb{R}^n,g_E,\frac{x^2}{4}\right)$. 
\end{theorem}
\begin{proof}
The inequality follows by (\ref{8.1}). For the rigidity, although the method is standard (cf. \cite[Corollary 8.17, p.392]{CCG1}), we briefly explain it here. By Proposition \ref{p6.1} and the construction of \(\ell_f^x\), the equality in (\ref{8.1}) is equivalent to  
\begin{align*}
	H(X,T;Y,S)= (4\pi(T-S))^{-n/2} e^{-\ell_{(X,T)}(Y,S)}
\end{align*}
for the associated Ricci flow and fixed $X,T,Y,S$. For our concern, we may choose $T=0$. Set $L(z,r) := (4\pi(T-r))^{-n/2} e^{-\ell_{(X,T)}(z,r)}$.
Recall that \(H\geq L\) and $\mathcal D := (\partial_r+\Delta_{g(r)}-R_{g(r)})L \geq 0$
in the distributional sense. For \(S<r<T\), define
\begin{align*}
	I(r) := \int_{M^n} L(z,r)\,H(z,r;Y,S)\,dv_{g(r)}(z).
\end{align*}
To prove $I$ is constant, fix \(S<a<b<T\) and set $u(z,r):=H(z,r;Y,S)$. For $\Lambda>0$, 
let \(\phi^\Lambda\) be the standard spacetime shrinker cutoff function
constructed in \cite[Lemma 3]{LW}. Thus 
\begin{align*}
	\frac{\left|\nabla_{g(r)}\phi^\Lambda\right|_{g(r)}^2}{\phi^\Lambda}
	\leq\frac{C(n)}{\Lambda},\qquad\left|\left(\partial_r-\Delta_{g(r)}\right)\phi^\Lambda\right|
	\leq\frac{C(n)}{\Lambda}.
\end{align*}
Define $I_\Lambda(r):=\int_{M^n}\phi^\Lambda(z,r)L(z,r)u(z,r)\,dv_{g(r)}(z)$.
Since $	\partial_rdv_{g(r)}=-R_{g(r)}\,dv_{g(r)}$ and $\left(\partial_r-\Delta_{g(r)}\right)u=0$,
compactly supported integration by parts gives
\begin{align*}
	&I_\Lambda(b)-I_\Lambda(a)\\
	&=\int_{M^n\times(a,b)}\phi^\Lambda u\,d\mathcal D
	+\int_a^b\int_{M^n}L\left[u\left(\partial_r-\Delta_{g(r)}\right)\phi^\Lambda-
	2\left\langle\nabla_{g(r)}u,\nabla_{g(r)}\phi^\Lambda\right\rangle_{g(r)}\right]dv_{g(r)}\,dr.
\end{align*}
On the compact time interval \([a,b]\), the heat kernel bounds and the standard energy estimate (cf. \cite[Theorem 7, 15, 16 and estimate (195)]{LW}) give
\begin{align*}
	\sup_{M^n\times[a,b]}L<\infty,
	\quad
	\int_a^b\int_{M^n}
	\left|\nabla_{g(r)}u\right|_{g(r)}^2
	\,dv_{g(r)}\,dr<\infty,
	\quad
	\int_{M^n}L(\cdot,r)\,dv_{g(r)}\leq1.
\end{align*}
The two cutoff-error terms are therefore respectively \(O(\Lambda^{-1})\) and \(O(\Lambda^{-1/2})\). Setting \(\Lambda\to\infty\) gives
\begin{align*}
	I(b)-I(a)=\int_{M^n\times(a,b)}u\,d\mathcal D\geq0.
\end{align*}
Moreover, we have $\lim_{r\downarrow S} I(r) = L(Y,S)$ and $\lim_{r\uparrow T} I(r) = H(X,T;Y,S)$.
Since these two limits agree by assumption, then \(I\) is constant. As the forward heat kernel is strictly positive and \(\mathcal D\) is a nonnegative distribution, it follows that $\mathcal D\equiv0$ on $M^n\times(S,T)$.
Hence \(L\) is a smooth solution of the conjugate heat equation by parabolic regularity. Its total mass is constant and tends to \(1\) as \(r\uparrow T\). Stochastic completeness also gives \(\int_{M^n} H(X,T;\cdot,\cdot)\,dv_{g(r)}=1\). Since \(0\leq L\leq H(X,T;\cdot,\cdot)\), we conclude that $L\equiv H(X,T;\cdot,\cdot)$ on $M^n\times(S,T)$.
Consequently, the reduced volume based at \((X,T)\) is identically equal to \(1\) for \(0<T-r<T-S\).

Write $\tau := T-r, g_\tau := g(T-\tau) = g(r)$
and set
\begin{align*}
	H_{X,T}(z,\tau)&:= H(X,T;z,T-\tau)= (4\pi\tau)^{-n/2} e^{-b_{X,T}(z,\tau)}.
\end{align*}
Since \(H= L\), we have $	b_{X,T}(z,\tau) = \ell_{X,T}(z,\tau)$.
The conjugate heat equation $\partial_\tau H_{X,T}=\Delta_{g_\tau} H_{X,T} - R_{g_\tau} H_{X,T}$
gives
\begin{align*}
	\partial_\tau b_{X,T}= \Delta_{g_\tau} b_{X,T}-\left| \nabla_{g_\tau} b_{X,T} \right|_{g_\tau}^2
	+ R_{g_\tau}- \frac{n}{2\tau}.
\end{align*}
On the other hand, the Hamilton–Jacobi identity for the reduced distance gives
\begin{align*}
	2\,\partial_\tau \ell_{X,T}+ \left| \nabla_{g_\tau} \ell_{X,T} \right|_{g_\tau}^2
	- R_{g_\tau}+ \frac{\ell_{X,T}}{\tau}= 0.
\end{align*}
Combining these two equations and $b_{X,T} = \ell_{X,T}$, we obtain
\begin{align*}
	P:=\tau\left(2\Delta_{g_\tau} b_{X,T}- \left| \nabla_{g_\tau} b_{X,T} \right|_{g_\tau}^2
	+ R_{g_\tau}\right)+ b_{X,T}- n= 0.
\end{align*}
Thus Perelman's differential Harnack density $v_{X,T}(z,\tau):=PK_{X,T}(z,\tau)=0$. Moreover,
Perelman's differential Harnack identity (cf.\cite[Theorem 9.1]{P}, see also Li-Wang\cite[Equation (101)]{LW}) reads
\begin{align*}
		\left( \partial_\tau - \Delta_{g_\tau} + R_{g_\tau} \right) v_{X,T}= -2\tau\,\left| \operatorname{Rc}_{g_\tau}+ \operatorname{Hess}_{g_\tau} b_{X,T}
		- \frac{1}{2\tau} g_\tau\right|_{g_\tau}^{2}K_{X,T}.
\end{align*}
Since \(v_{X,T}\equiv 0\) and \(K_{X,T} > 0\), it follows that for every \(0 < \tau < T-S\),
\begin{align}
	\operatorname{Rc}_{g_\tau}+ \operatorname{Hess}_{g_\tau} b_{X,T}(\cdot,\tau)
	= \frac{1}{2\tau} g_\tau\label{8.2*}
\end{align}
A conventional way to prove $R\equiv0$ is following  Yokota's $\mathcal L$-geodesic argument for the equality case of the reduced-volume monotonicity( cf. \cite{TY1, TY}). Here we give a more direct method.

Fix \(\tau_0\in(0,T-S)\), and choose a regular minimizing
\(\mathcal L\)-geodesic
\begin{align*}
	\gamma:[0,\tau_0]\longrightarrow M^n,\quad\gamma(0)=X.
\end{align*}
Set \(q:=\gamma(\tau_0)\). Since \(b_{X,T}=\ell_{X,T}\), the first
variation formula gives $\gamma'(\tau)=\nabla_{g_\tau}b_{X,T}\left(\gamma(\tau),\tau\right)$.
For each fixed \(\tau\), equation~\eqref{8.2*} implies
\begin{align*}
	d R_{g_\tau}=2Rc_{g_\tau}(\nabla_{g_\tau} b_{X,T},\cdot),\quad
	\Delta_{b_{X,T}}R_{g_\tau}=\frac1{\tau}R_{g_\tau}-2\left|Rc_{g_\tau}\right|_{g_\tau}^2,
\end{align*}
where $\Delta_{b_{X,T}}:=\Delta_{g_\tau}-\left\langle\nabla_{g_\tau} b_{X,T},\nabla_{g_\tau}\cdot\right\rangle_{g_\tau}$.
Since \(g_\tau=g(T-\tau)\), the scalar curvature satisfies
\begin{align*}
	\partial_\tau R_{g_\tau}=-\Delta_{g_\tau} R_{g_\tau}-2\left|Rc_{g_\tau}\right|_{g_\tau}^2.
\end{align*}
Consequently,
\begin{align*}
	\frac{d}{d\tau}
	R_{g_\tau}\left(\gamma(\tau),\tau\right)=
	\partial_\tau R_{g_\tau}+\left\langle\nabla_{g_\tau} R_{g_\tau},\gamma'(\tau)\right\rangle_{g_\tau}
	=-\Delta_{b_{X,T}}R_{g_\tau}-2\left|Rc_{g_\tau}\right|_{g_\tau}^2=-\frac1{\tau}
	R_{g_\tau}\left(\gamma(\tau),\tau\right).
\end{align*}
It follows that $\tau R_{g_\tau}\left(\gamma(\tau),\tau\right)$ is constant in \(\tau\). As \(\tau\downarrow0\), one has
\(\gamma(\tau)\to X\), and the smoothness of the Ricci flow near
\((X,T)\) implies that
\(R_{g_\tau}(\gamma(\tau),\tau)\) remains bounded. Therefore $R_{g_{\tau_0}}(q)=0$, then the strong maximum principle
implies \(R_g\equiv0\), and hence
\((M^n,g,f)\) is the Gaussian shrinker.
\end{proof}

Based on the representation (\ref{8.1}) of heat kernel lower bound, we are able to give a rougher but more specific lower bound. Put $\rho(t):=\sqrt{e^t-1}$ and $\theta(t):=\arctan\sqrt{e^t-1}$, then we have
\begin{proposition}
	For every \(0<\eta<1\) and every \(t>0\),
	\begin{align}	\label{eq:17}
		\bar\ell^x_f(y,t)
		\leq\left[\frac{1}{4\rho(t)\theta(t)}+\frac{(1-\eta)\theta(t)}{12\eta\,\rho(t)}\right]d^2(x,y)-
		\frac{(1-\eta)\theta(t)}{2\rho(t)}\left(f(x)+f(y)\right).
	\end{align}
	Moreover, for every \(0<\varepsilon<1\), there exists a \(t_\varepsilon>0\), depending only on \(\varepsilon\), such that
	\begin{align}	\label{eq:23}
		\bar\ell^x_f(y,t)\leq	\frac{d^2(x,y)}{(4-\varepsilon)t}-\frac{1-\varepsilon}{2}\bigl(f(x)+f(y)\bigr)
	\end{align}
	for all \(x,y\in M^n\) and \(0<t\le t_\varepsilon\). 
\end{proposition}
\begin{proof}
	Fix \(t>0\) and abbreviate \(\rho=\sqrt{e^t-1}\), \(\theta=\arctan\rho\). Define
	\begin{align*}
		q=q(s):=\frac{\arctan\sqrt{e^s-1}}{\theta},\quad 0\le s\le t.
	\end{align*}
	Then \(q(0)=0,\; q(t)=1\). A direct differentiation gives
	\begin{align*}
		\frac{d}{ds}\arctan\sqrt{e^s-1}=\frac{1}{2\sqrt{e^s-1}},
	\end{align*}
	and hence
	\begin{align*}
		\frac{dq}{ds}=\frac{1}{2\theta\sqrt{e^s-1}}.
	\end{align*}
	Let \(\gamma:[0,t]\to M^n\) be an a.e. smooth curve from \(x\) to \(y\), and put \(\sigma(q):=\gamma(s(q))\). Here $s(q)$ means the inverse function of $q(s)$. Then $ \sigma'(q)=\gamma'(s(q))s'(q)$, which implies
	\begin{align*}
		\gamma'(s)= \sigma'(q)\frac{dq}{ds}= \frac{\sigma'(q)}{2\theta\sqrt{e^s-1}}.
	\end{align*}
	Since \(ds = 2\theta\sqrt{e^s-1}\,dq\), then
	\begin{align*}
		\sqrt{e^s-1}\,|\dot\gamma(s)|^2\,ds
		&= \sqrt{e^s-1}\frac{|\sigma'(q)|^2}{4\theta^2(e^s-1)}\,
		2\theta\sqrt{e^s-1}\,dq= \frac{1}{2\theta}\left|\sigma'(q)\right|^2dq,
	\end{align*}
	and
	\begin{align*}
		-\frac{f(\gamma(s))}{\sqrt{e^s-1}}ds= -2\theta f(\sigma(q))dq.
	\end{align*}
	Therefore
	\begin{align*}
		\int_0^t\sqrt{e^s-1}\left(|\gamma'(s)|^2-\frac{f(\gamma(s))}{e^s-1}\right)ds=\frac{1}{2\theta}\int_0^1\left(\left|\sigma'(q)\right|^2-4\theta^2f(\sigma(q))
		\right)dq.
	\end{align*}
	Dividing by \(2\rho\), we obtain
	\begin{align}\label{eq:A7}
		\bar\ell^x_f(y,t)=\frac{1}{4\rho\theta}\inf_{\substack{\sigma(0)=x\\\sigma(1)=y}}\int_0^1\left(|\sigma'|^2-4\theta^2f(\sigma)\right)dq.
	\end{align}
	There exists a minimizing geodesic \(\sigma:[0,1]\to M^n\) with \(\sigma(0)=x,\;\sigma(1)=y\), parametrized at constant speed \(|\sigma'|\equiv d(x,y)\). Using this curve in \eqref{eq:A7} gives
	\begin{align}\label{eq:A8}
		\bar\ell^x_f(y,t)\leq\frac{d^2(x,y)}{4\rho\theta}- \frac{\theta}{\rho}\int_0^1f(\sigma(q))dq.
	\end{align}
	
	Since \(|\nabla f|^2\le f\), it follows that \(\sqrt f\) is \(1/2\)-Lipschitz, then
	along the minimizing geodesic, \(d(x,\sigma(q))=q\,d(x,y)\) and \(d\left(y,\sigma(q)\right)=(1-q)d(x,y)\). Thus
	\begin{align*}
		\sqrt{f(\sigma(q))}
		\geq\left(\sqrt{f(x)}-\frac{q\,d(x,y)}2\right)_+,\quad
		\sqrt{f(\sigma(q))}
		\geq
		\left(\sqrt{f(y)}-\frac{(1-q)d(x,y)}2\right)_+,
	\end{align*}
	and consequently,
	\begin{align*}
		f(\sigma(q))\geq\frac12\left[\left(\sqrt{f(x)}-\frac{q\,d(x,y)}2\right)_+^2
		+ \left(\sqrt{f(y)}-\frac{(1-q)d(x,y)}2\right)_+^2\right].
	\end{align*}
	For every \(r,z\ge0\) and every \(0<\eta<1\), one has
	\begin{align}	\label{eq:A13}
		(r-z)_+^2\geq(1-\eta)r^2 - \frac{1-\eta}{\eta}z^2.
	\end{align}
	Apply \eqref{eq:A13} to the two terms above to obtain
	\begin{align*}
		f(\sigma(q))\geq\frac{1-\eta}{2}\left(f(x)+f(y)\right)- \frac{1-\eta}{8\eta}\left(q^2+(1-q)^2\right)d^2(x,y).
	\end{align*}
	Integrating and using \(\int_0^1 q^2\,dq = \int_0^1(1-q)^2\,dq = \frac13\) gives
	\begin{align}\label{eq:A15}
		\int_0^1f(\sigma(q))\,dq\geq\frac{1-\eta}{2}\left(f(x)+f(y)\right)- \frac{1-\eta}{12\eta}\,d^2(x,y).
	\end{align}
	Substituting \eqref{eq:A15} into \eqref{eq:A8}yields exactly \eqref{eq:17}. 
	
	To prove $(\ref{eq:23})$, define
	\begin{align*}
		A_\eta(t):= \frac{1}{4\rho\theta} + \frac{(1-\eta)\theta}{12\eta\,\rho},\quad B_\eta(t):= \frac{(1-\eta)\theta}{2\rho}.
	\end{align*}
	Then (\ref{eq:17}) reads $\bar\ell^x_f(y,t) \leq A_\eta(t)d^2(x,y) - B_\eta(t)\left(f(x)+f(y)\right)$. As \(t\downarrow0\),
	\begin{align*}
		\rho(t)\theta(t)= t+\frac{t^2}{6}+O\left(t^3\right),\quad
		\frac{\theta(t)}{\rho(t)}= 1-\frac t3+O\left(t^2\right),
	\end{align*}
	thus
	\begin{align*}
		\lim_{t\downarrow0}tA_\eta(t)=\frac14,\quad
		\lim_{t\downarrow0}B_\eta(t)=\frac{1-\eta}{2}.
	\end{align*}
	Fix \(0<\varepsilon<1\) and choose \(\eta=\frac{\varepsilon}{2}\), then \(\frac14<\frac{1}{4-\varepsilon}\) and \(\frac{1-\eta}{2} > \frac{1-\varepsilon}{2}\). By continuity, there exists a \(t_\varepsilon>0\) such that for \(0<t\le t_\varepsilon\),
	\begin{align*}
		A_\eta(t) \le \frac{1}{(4-\varepsilon)t},\quad
		B_\eta(t) \ge \frac{1-\varepsilon}{2}.
	\end{align*}
	Since \(f(x)+f(y)\ge0\), \eqref{eq:17} gives \eqref{eq:23}.
\end{proof}

\section{Differential Harnack inequality}
Recall 
\begin{align*}
	w_f(x,y,t)=\tau\left(n-f(y)-\frac{2\Delta_{f,y}H_f}{H_f}+\frac{\left|\nabla_y H_f \right|^2}{H_f^2}\right)+f(y)-\ln H_f-\frac{n}{2}\ln4\pi\tau-n, \tau=1-e^{-t},
\end{align*}
and
$$\mathcal{W}_f(x,t)=\mathcal{W}\left(g,\sqrt{H_f}e^{-\frac{f(y)}{2}}, 1-e^{-t}\right)=\int_{M^n}w_f dv_{H_f}(y).$$
Due to the arguments in section \ref{sec-N}, $w_f$ is just the diffeomorphic counterpart of
\begin{align*}
	w_{x,t}(y,s)=(t-s)\left(2\Delta_{g_s} b_{x,t}(y,s)-\left|\nabla_{g_s} b_{x,t}(y,s)\right|_{g_s}^{2}
	+R_{g_s}(y)\right)+b_{x,t}(y,s)-n.
\end{align*}
By Li-Wang \cite[Theorem 21]{LW}, we know $w_{x,t}\leq0$, and hence $w_f\leq0$.

Let $\tilde{H}_f(x,y,t)=H_f(x,y,t)e^{-f(y)}=\frac{1}{\left[4\pi(1-e^{-t})\right]^{\frac{n}{2}}}e^{-b_f(x,y,t)}$, then $H_f(x,y,t)=\tilde{H}_fe^f$. We compute
\begin{align*}
	\Delta_fH_f&=\Delta_f\left(\tilde{H}_fe^f\right)\\
	&=\tilde{H}_fe^f\left(\Delta_ff+\left|\nabla f\right|^2\right)+e^f\Delta_f\tilde{H}_f+2e^f\left\langle \nabla \tilde{H}_f,\nabla f\right\rangle\\
	&=\tilde{H}_fe^f\left(\frac{n}{2}-R\right)+e^f\Delta\tilde{H}_f+e^f\left\langle \nabla \tilde{H}_f,\nabla f\right\rangle,
\end{align*}
and
\begin{align*}
	\left|\nabla H_f\right|^2=\left|\nabla \left(\tilde{H}_fe^f\right)\right|^2=e^{2f}\left(\left|\nabla \tilde{H}_f\right|^2+\tilde{H}_f^2\left|\nabla f\right|^2\right)+2\tilde{H}_fe^{2f}\left\langle \nabla \tilde{H}_f,\nabla f\right\rangle.
\end{align*}
Therefore the differential Harnack inequality $w_f\leq0$ implies
\begin{align*}
	&w_f(x,y,t)\\
	&=\left(1-e^{-t}\right)\left(n-f(y)-\frac{2\Delta_{f,y}H_f}{H_f}+\frac{\left|\nabla_y H_f \right|^2}{H_f^2}\right)+f(y)-\ln H_f-\frac{n}{2}\ln4\pi\left(1-e^{-t}\right)-n\\
	&=\left(1-e^{-t}\right)\left(R(y)-\frac{2\Delta_y\tilde{H}_f}{\tilde{H}_f}+\frac{\left|\nabla_y \tilde{H}_f \right|^2}{\tilde{H}_f^2}\right)-\ln \tilde{H}_f-\frac{n}{2}\ln4\pi\left(1-e^{-t}\right)-n\\
	&=\left(1-e^{-t}\right)\left(R(y)-\frac{2\Delta_y\tilde{H}_f}{\tilde{H}_f}+\frac{\left|\nabla_y \tilde{H}_f \right|^2}{\tilde{H}_f^2}\right)+b_f-n\leq0,
\end{align*}
and this gives
\begin{align}\label{11.1}
	R(y)\leq \frac{n-b_f}{1-e^{-t}}+\frac{2\Delta_y\tilde{H}_f}{\tilde{H}_f}-\frac{\left|\nabla_y \tilde{H}_f \right|^2}{\tilde{H}_f^2}.
\end{align}
In view of Theorem \ref{t3.8}, for fixed $(x,t)$, $\tilde{H}_f(x,y,t)=H_f(x,y,t)e^{-f(y)}\longrightarrow0$ uniformly as $y\longrightarrow\infty$, therefore there exists a maximum point $y_{x,t}$. Moreover, by Proposition \ref{p6.2}, we have $b_f(x,y_{x,t},t)\geq\boldsymbol{\mu}$. Then by (\ref{11.1}), we have
\begin{align}\label{11.2}
	R(y_{x,t})\leq \frac{n-\boldsymbol{\mu}}{1-e^{-t}}.
\end{align}
On the other hand, by Proposition \ref{p7.6} with appropriate $\beta=\beta(\varepsilon)$, we have
\begin{align}
	&H_f(x,y_{x,t},t)e^{-f(y_{x,t})}\nonumber\\
	&=H_\phi(x,y_{x,t},t)\exp\left(\frac{f(x)+f(y_{x,t})}{2}-f(y_{x,t})\right)\nonumber\\
	&\leq C(n,\varepsilon)e^{-\boldsymbol{\mu}}\Gamma(t) \exp \left(\frac{f(x)+f(y_{x,t})}{2} \left(1-\frac{\tanh \frac{t}{4}}{2}\right)-f(y_{x,t})-\frac{d^2(x,y_{x,t})}{(8+\varepsilon)t}\right).\label{9.3}
\end{align}

As before, write $H_f(x,y,t)e^{-f(y)}=\frac{e^{-b_f(x,y,t)}}{\left(4\pi\left(1-e^{-t}\right)\right)^{\frac{n}{2}}}$, recall by the definition, we have
\begin{align*}
	\mathcal{N}_f(x,t)&=\int_{M^n}b_f(x,y,t)H_f(x,y,t)e^{-f(y)}dv(y)-\frac{n}{2}\\
	&\geq \min{b_f(x,y,t)}\int_{M^n}H_f(x,y,t)e^{-f(y)}dv(y)-\frac{n}{2}\\
	&\geq \min{b_f(x,y,t)}-\frac{n}{2}.
\end{align*}
Since $\mathcal{N}_f\leq0$, the above estimate implies
\begin{align*}
	\frac{e^{-\frac{n}{2}}}{\left(4\pi\left(1-e^{-t}\right)\right)^{\frac{n}{2}}}\leq H_f(x,y_{x,t},t)e^{-f(y_{x,t})}.
\end{align*}
Comparing this estimate with (\ref{9.3}) implies
\begin{align*}
	\boldsymbol{\mu}-C(n,\varepsilon)\leq\frac{f(x)+f(y_{x,t})}{2}\left(1-\frac{\tanh \frac{t}{4}}{2}\right)-f(y_{x,t})-\frac{d^2(x,y_{x,t})}{(8+\varepsilon)t},
\end{align*}
which gives
\begin{align}
	d^2(x,y_{x,t})&\leq(8+\varepsilon)t\left(\left(\frac{2-\tanh \frac{t}{4}}{4}\right) f(x)-\frac{2+\tanh\frac{t}{4}}{4}f(y_{x,t})-\boldsymbol{\mu}+C(n,\varepsilon)\right)\label{9.4*}\\
	&\leq (8+\varepsilon)t\left(\left(\frac{2-\tanh \frac{t}{4}}{4}\right) f(x)-\boldsymbol{\mu}+C(n,\varepsilon)\right).\nonumber
\end{align}	
For example, if we choose $t'=t'(n,\varepsilon,\boldsymbol{\mu})$ small enough, then
\begin{align*}
	d^2\left(x,y_{x,t'}\right)\leq \frac{f(x)}{100}+C(n,\varepsilon)-\boldsymbol{\mu},
\end{align*}
and by (\ref{11.2}), 
\begin{align*}
	R\left(y_{x,t'}\right)\leq \frac{n-\boldsymbol{\mu}}{1-e^{-t'}}\leq C(n,\varepsilon,\boldsymbol{\mu}).
\end{align*}	
This implies that for any $x\in M^n$, we can find a point with uniformly bounded scalar curvature within $\sqrt{f(x)/C}$ scale.

Especially, we have the following distance control of low scalar curvature points
\begin{proposition}
	Let \((M^{n},g,f)\) be a complete non-compact Ricci shrinker and fix \(\delta>0\). For \(F:=f(x)\) and set \(A=A(F)\) be any positive function satisfying
	\begin{align}
		A(F)\longrightarrow\infty,\quad A(F)=o(F).\label{eq:I-Acond}
	\end{align}
	Then, for all $x\in M^n $ satisfying $f(x)=F\geq \underline{F}=\underline{F}\left(n,\delta,\boldsymbol{\mu},A\right)$, there exists a point $y$ such that
	\begin{align}
		d(x,y)\leq C(n,\delta,\boldsymbol{\mu})\frac{A(F)}{\sqrt F},\quad
		R(y)\leq2(n-\boldsymbol{\mu})\frac{F}{A(F)}.\label{eq:I-main}
	\end{align}
	Consequently, \(R(y)=o(F)=o(f(x))\). In particular, the scale produced by this argument can be taken to be
	\begin{align*}
		r(F)=\frac{A(F)}{\sqrt F},\quad
		A(F)\to\infty
		\text{ arbitrarily slowly}.
	\end{align*}
\end{proposition}

\begin{proof}
	Set \(K:=8+\delta\), \(B:=C_{0}(n,\delta)\left(1+|\boldsymbol{\mu}|\right)\), and
	\begin{align*}
		a(t):=\tanh\frac t4,\quad\alpha(t):=\frac{2-a(t)}4,\quad\beta(t):=\frac{2+a(t)}4,
	\end{align*}
	then by (\ref{9.4*}), for $G:=f(y_{x,t})$ and $r:=d(x,y_{x,t})$,
	\begin{align}
		r^{2}\leq Kt\left(\alpha(t)F-\beta(t)G+B\right).	\label{eq:I-r2}
	\end{align}
	Note that \(\alpha(t)-\beta(t)=-\frac12a(t)\le0\). Choose $t=t_F:=\frac{A(F)}F$,
	which tends to \(0\) by (\ref{eq:I-Acond}). Let \(z:=y_{x,t_F}\). For \(0<t\le1\), we have $	1-e^{-t}\ge\frac t2$, hence from (\ref{11.2})
	\begin{align}\label{9.9**}
		R(z)\le 2(n-\boldsymbol{\mu})\frac{F}{A(F)}.
	\end{align}
	Discarding \(-\beta G\) in (\ref{eq:I-r2}) and using \(B\le F/2\) for large \(F\) gives \(r\le\sqrt{KA(F)}\). Since \(A(F)=o(F)\), \(r<2\sqrt F\) for large \(F\).
	
	Using the \(1/2\)-Lipschitz property of \(\sqrt f\), we have $\sqrt G\ge \sqrt F-\frac r2$,
	and thus
	\begin{align*}
		G\ge F-\sqrt F\,r+\frac{r^{2}}4.
	\end{align*}
	Substituting this into (\ref{eq:I-r2}) and discarding non-positive terms yields
	\begin{align*}
		r^{2}\leq\frac{3K}{4}t_F\sqrt F\,r + KBt_F.
	\end{align*}
	Applying the elementary inequality \(r^{2}\le pr+q\Rightarrow r\leq p+\sqrt q\) with $p,q\geq0$ gives
	\begin{align*}
		r\leq\frac{3K}{4}t_F\sqrt F + \sqrt{KBt_F}= \frac{3K}{4}\frac{A(F)}{\sqrt F} + \sqrt{KB}\frac{\sqrt{A(F)}}{\sqrt F}.
	\end{align*}
	Since \(\sqrt{A(F)}\le A(F)\) for large \(F\), we obtain
	\begin{align}\label{9.9*}
		d(x,z)\le C\frac{A(F)}{\sqrt F},
	\end{align}
	where \(C=\frac{3(8+\delta)}4+\sqrt{(8+\delta)C_0(n,\delta)\left(1+|\boldsymbol{\mu}|\right)}\). This completes the proof of (\ref{eq:I-main}).
\end{proof}

As a direct corollary, we have

\begin{corollary}[\textbf{=Theorem \ref{t1.4'}}]
	Let $(M^{n},g,f)$ be a noncompact Ricci shrinker. Assume that $|\nabla R|=o\left(f^{3/2}\right)$ uniformly at infinity, i.e.
	\begin{align}
		\eta(L):=\sup_{\{z\in M^n:\,f(z)\ge L\}}\frac{|\nabla R|(z)}{f^{3/2}(z)}\longrightarrow0
		\quad\text{as}\quad L\longrightarrow\infty.\label{eq:II-eta}
	\end{align}
	Then
	\begin{align*}
		\frac{R(x)}{f(x)}\longrightarrow0\quad\text{as }f(x)\longrightarrow\infty.
	\end{align*}
\end{corollary}

\begin{proof}
	Let $A(F):=\min\bigl\{F^{1/4},\,\eta(F/2)^{-1/2}\bigr\}$
	with the convention \(\eta(F/2)^{-1/2}=+\infty\) if \(\eta(F/2)=0\), where $\eta$ is defined by (\ref{eq:II-eta}). Then \(A(F)\to\infty\), \(A(F)\le F^{1/4}\), and
	\begin{align}
		A(F)\eta(F/2)\longrightarrow0.\label{eq:II-Aeta}
	\end{align}
	As before, set \(t=t_F:=A(F)/F\) and \(y:=y_{x,t_F}\).
	Let \(\gamma\) be a minimizing geodesic from \(x\) to \(y\), Then (\ref{9.9*}) implies $f\left(\gamma(s)\right)/F\to1$ uniformly, i.e. for large \(F\),
	\begin{align}
		\frac F2\le f(\gamma(s))\le 2F.\label{eq:II-fbound}
	\end{align}
	Definition of \(\eta\) and (\ref{eq:II-fbound}) gives $\left|\nabla R\right|(\gamma(s))\le 2^{3/2}\eta(F/2)F^{3/2}$.
	Integrating along \(\gamma\) and by (\ref{9.9*}), then for $r:=d(x,y_{x,t_F})$, we have
	\begin{align*}
		\left|R(x)-R(y)\right|\leq r\cdot 2^{3/2}\eta(F/2)F^{3/2}\leq 2^{3/2}CA(F)\eta(F/2)F.
	\end{align*}
	Therefore by (\ref{eq:II-Aeta}),
	\begin{align}
		\frac{\left|R(x)-R(y)\right|}{F}\leq 2^{3/2}CA(F)\eta(F/2)\longrightarrow0.\label{eq:II-diffR}
	\end{align}
	Combining (\ref{9.9**}) and (\ref{eq:II-diffR}),
	\begin{align*}
		\frac{R(x)}F\leq \frac{2(n-\boldsymbol{\mu})}{A(F)} + 2^{3/2}CA(F)\eta(F/2)\longrightarrow0.
	\end{align*}
\end{proof}	

Besides applications based on point-wise estimates of heat kernel, integral representation via heat kernel is another important direction. Here we only give a remark for this point of view.
\begin{remark}[A dimension-- and entropy--free gradient estimate for the scalar curvature]
	Suppose 
	\begin{align*}
		K:=\sup_{M^n}\left|Rc\right|<\infty.
	\end{align*}
	Let $dv_f:=e^{-f}\,dv$ and denote by \(P_t^f=e^{t\Delta_f}\) the weighted heat semigroup.
	Since
	\begin{align*}
		\Delta_fR=R-2\left|Rc\right|^2,
	\end{align*}
	the standard cutoff argument on Ricci shrinkers and the
	Duhamel formula give, for every \(T>0\),
	\begin{align*}
		R(x)=e^{-T}P_T^fR(x)+2\int_0^T e^{-t}P_t^f\left|Rc\right|^2(x)\,dt.
	\end{align*}
	Since \(0\leq R\leq\sqrt{n}K\) and \(P_T^f1=1\), we have
	\begin{align*}
		0\leq e^{-T}P_T^fR(x)\leq\sqrt{n}K e^{-T}.
	\end{align*}
	Setting \(T\to\infty\), we obtain
	\begin{align*}
		R(x)=2\int_0^\infty e^{-t}P_t^f\left|Rc\right|^2(x)\,dt.
	\end{align*}
	By extensions of Li--Wang\cite[Theorems 3.15 and 3.16]{LW2} based on Bamler's sharp gradient estimates \cite[Theorem 4.1, Proposition 4.2]{Ba}, one may repeat the directional \(p=1\) argument of \cite[Corollary~5.3]{Koi}. After pulling the resulting estimate back through Proposition~\ref{p6.1}, we have
	\begin{align*}
		\sqrt{e^t-1}\int_{M^n}\left|\nabla_v^xH_f(x,y,t)\right|\,dv_f(y)
		\leq\frac{\Gamma(1)}{\sqrt{\pi}}=\frac1{\sqrt{\pi}}
	\end{align*}
	for every unit vector $v\in T_xM^n$.
	Since \(0\leq \left|Rc\right|^2\leq K^2\), it follows that
	\begin{align*}
		\left|\nabla^x_vP_t^f\left|Rc\right|^2\right|(x)\leq K^2\int_{M^n}\left|\nabla_v^xH_f(x,y,t)\right|\,dv_f(y)
		\leq\frac{K^2}{\sqrt{\pi(e^t-1)}}.
	\end{align*}
	Since
	\begin{align*}
		\int_0^\infty\frac{e^{-t}}{\sqrt{e^t-1}}\,dt=\frac{\pi}{2},
	\end{align*}
then
	\begin{align*}
		\left|\nabla^x_vR\right|(x)\leq2\int_0^\infty e^{-t}\left|\nabla^x_vP_t^f\left|Rc\right|^2(x)\right|\,dt
		\leq\frac{2K^2}{\sqrt{\pi}}\int_0^\infty\frac{e^{-t}}{\sqrt{e^t-1}}\,dt=\sqrt{\pi}K^2.
	\end{align*}
	Therefore
	\begin{align*}
			\sup_{M^n}|\nabla R|\leq\sqrt{\pi}K^2.
	\end{align*}
\end{remark}

\noindent {\it\textbf{Acknowledgement}}: The authors are supported by YSBR--001, NSFC--12431003 and a research fund from USTC. The results and proofs presented in this paper were developed by the authors over the past two years, independently of AI tools. AI was used solely to assist in checking the proofs and editing the English. The authors take full responsibility for the mathematical content of the paper.

	\vspace{0.5in}
	
     Bing  Wang,  Institute of Geometry and Physics, and School of Mathematical Sciences, University of Science and Technology of China, Hefei 230026, China; Hefei National Laboratory, Hefei 230088, China.

Email: topspin@ustc.edu.cn
	
	Jie  Wang,  Institute of Geometry and Physics, University of Science and Technology of China,  Hefei 230026, China.
	
	Email: wangjie9math@163.com
	
\end{document}